\documentclass[12pt]{amsart}
\usepackage{amsthm, amstext, amsmath, amssymb, graphicx, verbatim, fullpage}
\usepackage{enumerate}
\usepackage{graphicx}
\usepackage{amssymb,amsmath, amsthm}
\usepackage{booktabs}
\usepackage{hyperref}
\usepackage{ mathrsfs }
\usepackage{algorithm}
\usepackage{xcolor}
\usepackage{versions}

\usepackage[normalem]{ulem} 

\usepackage[utf8]{inputenc}
\usepackage[english]{babel}

\newtheorem{thm}{Theorem}[section]

\newtheorem{lem}[thm]{Lemma}
\newtheorem{prop}[thm]{Proposition}

\theoremstyle{definition}
\newtheorem{defn}[thm]{Definition}
\newtheorem{notation}[thm]{Notation}
\newtheorem{question}[thm]{Question}

\newtheorem{rmk}[thm]{Remark}

\def\ZZ{\mathbb{Z}}
\def\QQ{\mathbb{Q}}

\def\CC{\mathbb{C}}
\def\PP{\mathbb{P}}
\def\FF{\mathbb{F}}
\def\Fp{\mathbb{F}_p}

\newcommand{\dd}[2]{\frac{\partial #1}{\partial #2}}

\newcommand{\A}{\mathbb{A}}

\newcommand{\N}{\mathbb N}

\newcommand{\C}{\mathbb C}
\newcommand{\Q}{\mathbb Q}

\newcommand{\M}{\mathcal{M}}
\newcommand{\cA}{\mathcal{A}}

\DeclareMathOperator{\lcm}{lcm}
\DeclareMathOperator{\Aut}{Aut}
\DeclareMathOperator{\Conj}{Conj}
\DeclareMathOperator{\Fix}{Fix}
\DeclareMathOperator{\Per}{Per}
\DeclareMathOperator{\PGL}{PGL}
\DeclareMathOperator{\PSL}{PSL}

\DeclareMathOperator{\Proj}{Proj}
\def\PSLp{\PSL_2(\Fp)}

\def\PSLq{\PSL_2(\FF_q)}
\def\PGLq{\PGL_2(\FF_q)}
\newcommand{\Fpbar}{{\bar{\FF}}_p}
\newcommand{\PGLpbar}{\PGL_2(\Fpbar)}

\DeclareMathOperator{\SL}{SL}
\DeclareMathOperator{\GL}{GL}
\DeclareMathOperator{\Spec}{Spec}
\DeclareMathOperator{\Rat}{Rat}

\excludeversion{code}

\title{Automorphism Groups of Endomorphisms of $\mathbb{P}^1 (\bar{\mathbb{F}}_p)$}

\author{Julia Cai}
\address{Department of Mathematics,
         Yale University.
         New Haven, CT 06520}
\email{julia.cai@yale.edu}

\author{Benjamin Hutz}
\address{Department of Mathematics and Statistics,
         Saint Louis University,
         St.~Louis, MO 63103}
\email{benjamin.hutz@slu.edu}

\author{Leo Mayer}
\address{Department of Mathematics,
         Lawrence University,
         Appleton, WI 54911}
\email{leopold.e.mayer@lawrence.edu}

\author{Max Weinreich}
\address{Department of Mathematics,
         Brown University,
        Providence, RI 02912}
\email{max@math.brown.edu}



\thanks{Max Weinreich was supported by a National Science Foundation Graduate Research Fellowship under Grant No. 2040433.}

\begin{document}

\maketitle
\begin{abstract}
    For any algebraically closed field $K$ and any endomorphism $f$ of $\mathbb{P}^1(K)$ of degree at least 2, the automorphisms of $f$ are the M\"obius transformations that commute with $f$, and these form a finite subgroup of $\operatorname{PGL}_2(K)$. In the moduli space of complex dynamical systems, the locus of maps with nontrivial automorphisms has been studied in detail and there are techniques for constructing maps with prescribed automorphism groups that date back to Klein. We study the corresponding questions when $K$ is the algebraic closure $\bar{\mathbb{F}}_p$ of a finite field.  We use the classification of finite subgroups of $\operatorname{PGL}_2(\bar{\mathbb{F}}_p)$ to show that every finite subgroup is realizable as an automorphism group. To construct examples, we use methods from modular invariant theory. Then, we calculate the locus of maps over $\bar{\mathbb{F}}_p$ of degree $2$ with nontrivial automorphisms, showing how the geometry and possible automorphism groups depend on the prime $p$.
\end{abstract}

\section{Introduction}

Let $K$ be an algebraically closed field. A dynamical system of degree $d$ on the projective line is an endomorphism of $\PP^1(K)$ and can be represented in coordinates as a pair of homogeneous polynomials of degree $d$ with coefficients in $K$ and no common factors. We assume throughout that $d \geq 2$. The set of all such dynamical systems is denoted $\Rat_d$. There is a natural conjugation action on $\Rat_d$ by automorphisms of $\PP^1$, the group $\PGL_2$, given as
\begin{equation*}
    f^{\alpha} = \alpha^{-1} \circ f \circ \alpha \quad \text{for $f \in \Rat_d$ and $\alpha \in \PGL_2$.}
\end{equation*}
The quotient by this action, see Silverman \cite{Silvermana}, is the \emph{moduli space of dynamical systems of degree $d$},
$$\M_d:= \Rat_d/\PGL_2.$$
We use square brackets to distinguish between a map $f$ in $\Rat_d$ and its conjugacy class $[f]$ in $\M_d$.
An \emph{automorphism} (or \emph{symmetry}) of $f$ is an element $\alpha$ of $\PGL_2(K)$ such that
$$ f^\alpha = f.$$
The set of such $\alpha$ is a subgroup of $\PGL_2(K)$, called the \emph{automorphism group} of $f$. We denote it $\Aut(f)$. Since these automorphisms have finite invariant sets of points, such as the periodic points of some fixed period, the automorphism group of a given map must be finite.

Our objects of study are those maps $f$ for which $\Aut(f)$ is nontrivial: that is, those $f$ which have an automorphism besides the identity. As is the case with elliptic curves that have complex multiplication, dynamical systems with nontrivial automorphisms can feature exceptional properties. For instance, a complex dynamical system with icosahedral symmetry was used to solve the quintic through iteration \cite{DM}.

We will need to know how conjugation affects automorphism groups. Given $\sigma \in \PGL_2$, the conjugation action on $\Aut(f)$ defined by $\alpha \mapsto \alpha^\sigma$ defines a group isomorphism
$$\Aut(f) \cong \Aut(f^\sigma).$$
The conjugacy class of $\Aut(f)$ in $\PGL_2$ is, thus, a well-defined invariant of $[f]$. When we speak of the automorphism group associated to $[f]$, we understand this group to be well-defined only up to conjugacy.

In particular, the locus of rational maps with a nontrivial automorphism group descends to a well-defined subset of $\M_d$. We call this set the \emph{automorphism locus of $\M_d$}, denoted $\cA_d$. Note that conjugation may affect the field of definition of both the map and its automorphism group, and determining the minimal field of definition of a conjugacy class and/or its automorphism group can often be a delicate question, e.g., \cite{Hutz15, Silverman12}.

In this article, we initiate the study of dynamical systems with nontrivial automorphisms over finite fields and their algebraic closures.
Specifically, we address the following pair of questions;
\begin{enumerate}
    \item How can we construct examples of dynamical systems over $\Fpbar$ with nontrivial automorphisms, and which automorphism groups can arise?
    \item What is the structure of the automorphism locus $\cA_2$ in the moduli space $\M_2 (\Fpbar)?$
\end{enumerate}
We fully resolve the realizability problem.
\begin{thm}\label{thm_0}
    Every finite subgroup of $\PGL_2(\Fpbar)$ occurs as the automorphism group of some dynamical system.
\end{thm}
We do not place restrictions on the degrees of the maps which realize the automorphism groups. However, in many cases, we prove that the given map has the smallest degree among all maps of degree $d \geq 2$ that realize a given automorphism group. We say such a map is \emph{of minimal degree} for that group. Explicit constructions and details are given in Theorem \ref{thm_realizable_p_irregular} and Theorem \ref{thm_realizable_p_regular}.

The methods used previously to construct dynamical systems with nontrivial automorphisms and to study automorphism loci depend on characteristic 0 in fundamental ways, opening the possibility that new phenomena emerge when we change the base field to a finite subfield of $\bar{\FF}_p$.
We investigate these new phenomena, emphasizing how our methods and results contrast with characteristic 0.

To provide context, we briefly describe some of what is known about $\cA_d$ in the complex case. As mentioned earlier, the automorphism group is a finite subgroup of $\PGL_2$, so the classification of such subgroups is important. In characteristic $0$, the finite subgroups of $\PGL_2$ were classified classically. For a modern exposition, see \cite{Silverman12}.

\begin{notation}
We set notation for referring to various groups.
\begin{itemize}
    \item Let $1$ denote the trivial group.
    \item Let $C_n$ denote the cyclic group of $n$ elements, for each $n \geq 2$.
    \item Let $D_{2n}$ denote the dihedral group of $2n$ elements, for each $n \geq 2$.
    \item Let $A_4$ denote the tetrahedral group.
    \item Let $S_4$ denote the octahedral group.
    \item Let $A_5$ denote the icosahedral group.
\end{itemize}
\end{notation}

The above are a complete list of finite subgroups of $\PGL_2(\C)$, up to conjugacy. The general problem of which subgroups of $\PGL_2(\C)$ can be realized as an automorphism group for some $f \in \Rat_d$ relies on tools from the classical invariant theory of finite groups; see \cite{Hutz15}, as well as partial results found in a number of other places, such as \cite{Silverman12}.

The problem of determining the locus $\cA_d(\C)$ has been studied in a number of articles \cite{Fujimura, charzero, MSW, Milnor, West}. The automorphism locus $\cA_d(\C)$ forms a Zariski closed proper subset of $\M_d(\C)$. In fact, for $d > 2$, the automorphism locus coincides with the singular locus of $\M_d(\C)$ \cite{MSW}. The case $d = 2$ stands in contrast: Milnor showed that $\M_2(\C)$ is isomorphic as a variety to the affine plane $\A^2(\C)$, which is smooth, and that the automorphism locus $\cA_2 (\C)$ is a cuspidal cubic curve \cite{Milnor}. The points of $\cA_2 (\C)$ all have an automorphism group isomorphic to $C_2$, except at the cusp, where the automorphism group is isomorphic to the symmetric group $S_3$. The descriptions of $\cA_3(\C)$ and $\cA_4(\C)$ are more recent and more complicated \cite{charzero, West}. The best results currently available for $\cA_d(\C)$ with $d \geq 5$ mostly focus on the dimensions of the various components \cite{MSW}.

We first study which automorphism groups are realizable. Among the finite subgroups $\Gamma$ of $\PGL_2$, which arise as automorphism groups of rational maps? We call this question the \emph{realizability problem} for $\Gamma$. If $\Gamma$ is realizable, so are its conjugates; thus, it suffices to look at one representative per conjugacy class. Our next theorems construct solutions to the realizability problem for every finite subgroup $\Gamma$ of $\PGL_2(\bar{\FF}_p)$.

We first review what is known in the complex case.
Miasnikov-Stout-Williams \cite{MSW} give the dimensions of the components of $\cA_d(\CC)$ associated to each finite $\Gamma \subset \PGL_2(\CC)$. They do not, however, give any explicit realizations or explore arithmetic questions, such as the necessary field of definition. The strongest results in this direction come from deFaria-Hutz \cite{Hutz15}.
They prove that every finite subgroup of $\PGL_2(\CC)$ is realizable as a subgroup of the automorphism group infinitely often (allowing the degree of the map to increase). This construction is explicit and relies on the classical invariant theory of finite groups.

In characteristic $p>0$, much less is known. While the classification of finite subgroups of $\PGL_2 (\Fpbar)$ is classical, the unpublished version by Faber \cite{Faber} in modern notation is the most readable. For each prime $p$, each conjugacy class for each subgroup supplies a case of the realizability problem. We summarize the classification in Proposition \ref{prop_faber}.

\begin{defn}
A finite subgroup of $\PGL_2(\Fpbar)$ is called \emph{$p$-regular} if $p$ does not divide the group order; otherwise, it is called \emph{$p$-irregular.}
\end{defn}

\begin{defn}
For each power $q$ of a prime $p$, the \emph{Borel group} $B(\FF_q)$ is the group of upper triangular matrices in $\PGL_2(\FF_q)$.
A \emph{$p$-semi-elementary} group is one that is the semi-direct product of a Sylow $p$-subgroup of order $p$ and a cyclic subgroup.
\end{defn}

\begin{prop}[{{Faber \cite{Faber}}}]\label{prop_faber}
Let $p$ be a prime. Each finite subgroup $\Gamma$ of $\PGL_2(\bar{\FF}_p)$ belongs to one of the following isomorphism types:
\begin{itemize}
    \item The identity group $1$;
    \item The cyclic group $C_n$, for each $n \geq 2$;
    \item The dihedral group $D_{2n}$, for each $n \geq 2$;
    \item The tetrahedral group $A_4$;
    \item The icosahedral group $A_5$;
    \item The octahedral group $S_4$;
    \item The group $\PGL_2(\FF_q)$, for some power $q$ of $p$;
    \item The group $\PSL_2(\FF_q)$, for some power $q$ of $p$;
    \item A $p$-semi-elementary group conjugate to a subgroup of the Borel group $B(\FF_q)$, for some power $q$ of $p$.
\end{itemize}
Except for $p$-semi-elementary groups, each possible isomorphism type occurs as at most one conjugacy class in $\PGL_2(\bar{\FF}_p)$.

For each power $q$ of $p$, each subgroup of $B(\FF_q)$ is of the form
$$ \{z \mapsto \alpha z + \beta : \; \alpha \in \mu, \; \beta \in \Lambda\}, $$
where $\mu$ is a subgroup of $\FF^\times_q$ of some order $n$, and $\Lambda$ is a subgroup of $\FF^+_q$ such that $\mu(\Lambda) \subseteq \Lambda$. A subgroup of $\PGL_2(\FF_q)$ is $p$-semi-elementary if and only if it is conjugate to a subgroup of $B(\FF_q)$ for which $\Lambda \neq 0$.
\end{prop}
Not every group named in Proposition \ref{prop_faber} appears for every prime, and for some small primes, there are accidental isomorphisms between some of the possible groups. The precise classification of subgroups of $\PGL_2(\bar{\FF}_p)$ up to conjugacy is given in the Appendix.

The next two theorems resolve the realizability question for $p$-irregular and $p$-regular subgroups, respectively. Together, the theorems show by explicit constructions that every finite subgroup of $\PGL_2$ arises as an automorphism group (Theorem \ref{thm_0}). For certain groups, we show that our constructions furnish maps which are of minimal degree among all maps with the prescribed automorphism group.

\begin{thm} \label{thm_realizable_p_irregular}
    Let $p$ be a prime, and let $q$ be a power of $p$. Let $\Gamma$ be a finite $p$-irregular subgroup of $\PGL_2(\bar{\FF}_p)$. Then there exists a rational map $f: \PP^1(\Fpbar)\to\PP^1(\Fpbar)$ with $\Aut(f) = \Gamma$. In particular, such a map $f$ can be constructed for each $\Gamma$ as follows.
    \begin{enumerate}
        \item \label{thm_realizable_pgl}
        Let $f(z) = z^q$. Then $\Aut(f) = \PGL_2(\FF_q)$, and $f$ is of minimal degree for $\PGL_2(\FF_q)$.

        \item \label{thm_realizable_semielementary}
            Let $\Gamma$ be a $p$-semi-elementary subgroup with associated additive group $\Lambda$ and integer $n$ in the form of Proposition \ref{prop_faber}. Then
            $$ f(z) = \prod_{\lambda \in \Lambda} (z - \lambda)^{n+1} + z $$
            satisfies $\Aut (f) = \Gamma$. In particular, if $\Lambda = \FF_q$, then
            $$ f(z) = {(z^q - z)}^{n + 1} + z.$$
        \item \label{thm_realizable_psl}
            If $p > 2$, then $\PSL_2(\FF_q)$ is distinct from $\PGL_2(\FF_q)$. In this case, there exists a map $f$ such that
            $$\Aut(f) = \PSL_2(\FF_q).$$
            We construct such an $f$ of degree $\frac{1}{2}(q^3 - 2q^2 + q + 2)$. Consider the two fundamental invariants of $\SL_2(\FF_q)$:
            \begin{align*}
                u &= x^qy - xy^q,\\
                c_1 &= \sum_{n = 0}^{q} x^{(q-1)(q-n)}y^{(q-1)n}.
            \end{align*}
            Also set
            $$a = \frac{q(q-3) + 4}{2}, \qquad b = \frac{q-1}{2}. $$
            Then take $f$ to be the dynamical system that arises from the Doyle-McMullen construction \eqref{eq_DM} applied to $F=c_1^b$ and $G=u^a$; that is,
            \begin{equation*}
                f(x,y) = \left[ x c_1^b + \frac{\partial u^a}{\partial y} : y c_1^b - \frac{\partial u^a}{\partial x}\right].
            \end{equation*}
            This $f$ is of minimal degree for $\PSL_2(\FF_q)$.
        \item \label{thm_realizable_irreg_dih_2}
        Let $p = 2$, and let $n \geq 3$ be odd. Then $f(z) = 1/z^{2n-1}$ has $\Aut(f) \cong D_{2n}$.
        \item \label{thm_realizable_irreg_a5}
            Let $p = 3$. There is a unique $p$-irregular subgroup of $\PGL_2(\bar{\FF}_3)$ isomorphic to $A_5$, up to conjugacy. There exists a map $f$ such that $\Aut(f) \cong A_5$. Specifically, there is a representation of $A_5$ in $\PGL_2(\bar{\FF}_3)$ with fundamental invariants
            \begin{align*}
                u_1 &= x^{10} + i y^{10},\\
                u_2 &= x^{11}y + (i+2) x^6 y^6 - i xy^{11},
            \end{align*}
            where $i \in \bar{\FF}_3$ satisfies $i^2 + 1 = 0.$
            Let $f$ be the dynamical system arising from the Doyle-McMullen construction \eqref{eq_DM} applied to $F = u_1^2$ and $G = u_1 u_2$, that is,
            \begin{equation*}
                f(x,y) = \left[ x u_1^2 - \dd{(u_1 u_2)}{y} : y u_1^2 + \dd{(u_1 u_2)}{x} \right].
            \end{equation*}
            Then $f$ has degree 21, and is of minimal degree for $A_5$ in $\PGL_2(\bar{\FF}_3)$.
        \end{enumerate}
\end{thm}

\begin{thm}\label{thm_realizable_p_regular}
    Let $p$ be a prime and $q$ a power of $p$.
    Let $\Gamma$ be a $p$-regular subgroup of $\PGL_2(\FF_q)$. Then there exists a rational map $f: \PP^1(\Fpbar)\to\PP^1(\Fpbar)$ with automorphism group exactly $\Gamma$. In particular, such a map $f$ can be constructed for each $\Gamma$ as follows.
    \begin{enumerate}
        \item The map $f(z) = z^2 + z$ has $\Aut(f) = 1$, and $f(z)$ is trivially of minimal degree for $\Gamma=1$. \label{p_regular_trivial}
        \item Let $n \geq 2$ be relatively prime to $p$. Then the map $f(z)=\frac{1}{z^{n-1}}+z$ has $\Aut(f)\cong C_n$. Furthermore, this map is of minimal degree for $C_n$. \label{p_regular_cyclic}
        \item Let $p > 2$ be prime and let $n \geq 2$ be coprime to $p$. The realizability problem for $D_{2n}$ over $\PGLpbar$ is solvable through one of the following constructions.\label{p_regular_dihedral}
        \begin{itemize}
            \item If $n \not\equiv -1 \mod p$, then the map $f(z) = z^{n + 1}$ has exact automorphism group $D_{2n}$. 
            \item If $n \not\equiv 1 \mod p$ and $n > 2$, then the map $f(z) = \frac{1}{z^{n - 1}}$ has exact automorphism group $D_{2n}$. This example is of minimal degree for $D_{2n}$.
            \item If $n = 2$, then for every $a \in \bar{\FF}_p$ not in the exceptional set $\{-3, -1, 0, 1 \}$, the map
            $$ f(z) = z \cdot \frac{z^2 + a}{a z^2 + 1}$$
            has $\Aut(f) \cong D_4$, and $f(z)$ is of minimal degree for $D_4$.
        \end{itemize}
        \item The tetrahedral group $A_4$ is realizable as an automorphism group of a degree 3 map over $\Fpbar$, for all $p \geq 5$, and 3 is the minimal degree for $A_4$. \label{p_regular_tetrahedral}
        \item The octahedral group $S_4$ is realizable as an automorphism group over $\Fpbar$, for all $p \geq 5$. \label{p_regular_octahedral}
        \item The icosahedral group $A_5$ is realizable as an automorphism group over $\Fpbar$, for all $p \geq 7$.\label{p_regular_icosahedral}
    \end{enumerate}
\end{thm}

The invariant theory constructions used in deFaria-Hutz \cite{Hutz15} go through in the $p$-regular case, but remain unknown in the modular case (where the characteristic $p$ divides the order of the group). Consequently, the methods used for our realizability results are a combination of adaptations of the invariant theory constructions and ad hoc computations. See the discussion at the beginning of Section \ref{sect_realizable}.

For the $p$-regular case in Theorem \ref{thm_realizable_p_regular}, we take maps in characteristic $0$ with the appropriate automorphism group and reduce modulo $p$; see Section \ref{sect_pregular}. The $p$-irregular case in Theorem \ref{thm_realizable_p_irregular} is more elaborate. The work of Klein \cite{klein_icosa} and Doyle-McMullen \cite{DM} shows that the problem of creating maps over $\C$ with prescribed automorphism group can be framed in terms of classical invariant theory. In the case of characteristic $p$ and a $p$-irregular group of automorphisms, we use modular invariant theory in place of classical invariant theory. Magma can calculate modular invariants \cite{magma}. By generating lots of invariants, we obtained a variety of maps which were candidates for realizing the subgroup in question. Throughout, there is the new difficulty that many maps with some prescribed automorphisms in fact have \emph{extra} automorphisms; that is, the automorphism group is all of $\PGL_2(\FF_q)$. We used the automorphism group calculation algorithm of Faber-Manes-Viray \cite{FMV}, which is implemented in Sage \cite{sage}, to check exactness of the automorphism groups. Examining the computational evidence, we were able to conjecture general forms for solutions and prove them. See Section \ref{sect_pirregular}.

We next study the locus of maps in $\cA_2 (\Fpbar)$ with a non-trivial automorphism.
For a given point $x \in \M_d$, we freely write $\Aut(x) \cong G$ to mean that any map representing $x$ has automorphism group isomorphic to $G$. Many subgroups of $\PGL_2$ arise in just one conjugacy class, so such a description often suffices to describe the conjugacy class $\Aut(x)$.

To state the result, we use the explicit isomorphism $\M_2 \to \A^2$ given by $f \mapsto (\sigma_1, \sigma_2)$, where $\sigma_1$ and $\sigma_2$ are the first two elementary symmetric polynomials evaluated at the multipliers of the fixed points of $f$. This isomorphism was established over $\C$ by Milnor \cite{Milnor} and extended to an isomorphism of schemes over $\Spec \ZZ$ by Silverman \cite[Theorem 5.1]{Silvermana}.

\begin{thm} \label{thm_Ad}
\hfill
The geometry of the automorphism locus $\cA_2(\bar{\FF}_p)$ depends on the prime $p$, in the following way.
\begin{enumerate}
    \item \label{thm1_1} \boxed{p = 2}: The automorphism locus $\mathcal{A}_2(\bar{\FF}_2)$ is the line $\sigma_1=0$.

    For every point $x=(\sigma_1,\sigma_2)$ except $(0,0)$ and $(0,1)$, we have
    $\Aut (x) \cong C_2$. For $x = (0,0)$, we have $\Aut (x) \cong S_3$ and for $x=(0,1)$ we have $\Aut(x)$ is trivial as a subgroup of $\PGL_2$ and isomorphic to $\alpha_2 \cong \bar{\FF}_2[t]/(t^2)$ as a group scheme.

    \item \label{thm1_2} \boxed{p = 3}: The automorphism locus $\mathcal{A}_2(\bar{\FF}_3)$ is the cuspidal cubic curve
        \begin{equation*}
            2\sigma_1^3 + \sigma_1^2\sigma_2 - \sigma_1^2 - \sigma_2^2 - 2\sigma_1\sigma_2=0.
        \end{equation*}
        Every point $x$ has
        $\Aut(x) \cong C_2$.\\
    \item \label{thm1_3} \boxed{p > 3}: The automorphism locus $\mathcal{A}_2(\bar{\FF}_p)$ is the cuspidal cubic curve
        \begin{equation*}
            2\sigma_1^3 + \sigma_1^2\sigma_2 - \sigma_1^2 - 4\sigma_2^2 - 8\sigma_1\sigma_2 + 12\sigma_1 + 12\sigma_2-36 = 0.
        \end{equation*}
        Every point $x$ except the cusp has
        $\Aut(x) \cong C_2$,
        and when $x$ is the cusp, we have
        $\Aut(x) \cong S_3$.
\end{enumerate}
\end{thm}

\begin{figure}[h]
    \begin{center}
        \caption{Geometry of $\cA_2(\Fpbar)$}
        \label{fig_geometry}
        \includegraphics[angle=0,scale=.74]{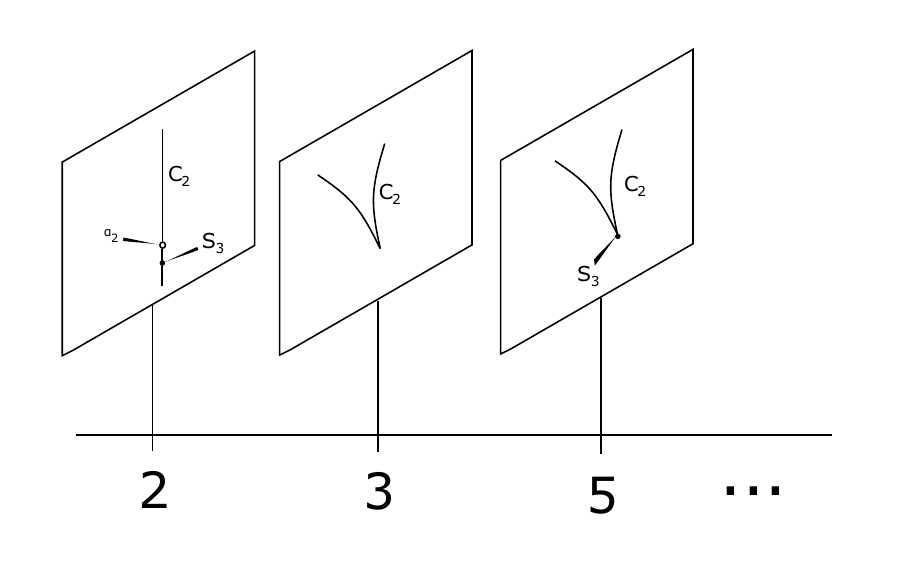}
    \end{center}
\end{figure}

We imagine this theorem in terms of the informal picture in Figure \ref{fig_geometry}. As $p$ varies, we obtain a family of curves.
Automorphism groups that were possible in characteristic 0 can collapse when we reduce modulo certain small primes. This kind of behavior is typical in arithmetic geometry.
More intriguing is, that without considering group schemes, the theorem over $\C$ that $\cA_d$ is Zariski closed fails in characteristic $p$. We can illustrate the phenomenon by the (dehomogenized) one-parameter family in $\Rat_2(\bar{\FF}_2)$ defined by
\begin{equation*}
    f_c(z) = z^2 + cz, \quad c \in \bar{\FF}_2.
\end{equation*}
We show in Section \ref{subsection_autlocusf2} that this family of rational maps forms a line in the moduli space and that the map $z \mapsto z + c - 1$ is an automorphism of $f_c$. This automorphism is nontrivial, unless $c = 1$, in which case the automorphism degenerates to the identity map. The reader can readily check that $\Aut(f_1)$ is trivial as a subgroup of $\PGL_2$. In Section \ref{subsection_autlocusf2}, we compute the automorphism group scheme \cite{FMV} of $f_1$ and find that it is the well-known group scheme $\alpha_2$. While the group of $\alpha_2$ is trivial, its group scheme structure is not.

\begin{question}
As the map $f_c$ varies, so does the nontrivial automorphism it carries. Can we create a moduli space that parametrizes rational maps with a choice of automorphism, and would the analogue of $\cA_d$ in this moduli space be a Zariski closed set? (This line of inquiry was suggested to us by Joseph Silverman.)
\end{question}

The structure of the article is as follows. In Section \ref{sect_realizable} we study the realizability problem and prove Theorem \ref{thm_realizable_p_irregular} and Theorem \ref{thm_realizable_p_regular}. This section starts with an introduction to the methods and proceeds through the cases of $p$-irregular followed by $p$-regular. In Section \ref{sect_dm}, we adapt the structure theorem of Doyle and McMullen \cite{DM} to the setting of modular invariant theory, and we prove that our example for $\PSL_2(\FF_q)$ is of minimal degree, Theorem \ref{thm_realizable_p_irregular}\eqref{thm_realizable_psl}. In Section \ref{sect_Ad} we study the structure of $\cA_d \subset \M_d$ and prove Theorem \ref{thm_Ad}.

The authors thank the Institute for Computational and Experimental Mathematical Research (ICERM) for hosting the summer Research Experience for Undergraduates (REU) program in 2019, where the majority of this work was completed. We also thank Xander Faber and Joe Silverman for helpful conversations and an anonymous referee for many helpful and detailed comments.


\section{Realizability} \label{sect_realizable}
In considering the realizability problem, our constructions are best understood in contrast to the resolution of the realizability problem over $\C$, which we sketch. This story spans centuries: It starts with Klein’s beautiful lectures on the icosahedron \cite{klein_icosa}, is continued in Doyle and McMullen’s work on the quintic \cite{DM}, and concludes in the recent paper by deFaria-Hutz \cite{Hutz15}.

If $f$ is a solution for the realizability problem for $\Gamma$, then for any $\sigma \in \PGL_2$, the conjugated map $\sigma^{-1} \circ f \circ \sigma$ is a solution for $\sigma^{-1} \Gamma \sigma$. So, to solve the realizability problem in general, we need only consider one representative of each conjugacy class of $\Gamma$ in $\PGL_2$. The finite subgroups of $\PGL_2(\C)$ were classified up to conjugacy by Klein \cite{klein_icosa}; a more modern version can be found in Silverman \cite{Silverman12}. The finite subgroups of $\PGL_2(\C)$ belong to one of the following isomorphism types:
\begin{itemize}
    \item a cyclic group $C_n$;
    \item a dihedral group $D_{2n}$;
    \item the tetrahedral group $A_4$;
    \item the octahedral group $S_4$;
    \item the icosahedral group $A_5$.
\end{itemize}
Each isomorphism type arises as just one conjugacy class in $\PGL_2(\C)$.

Klein’s strategy for creating maps with symmetry rested on what is now known as the classical invariant theory of finite groups. Roughly, classical invariant theory is an algorithm which takes as input a $\C$-vector space $V$ and a group representation $\Gamma \hookrightarrow \GL(V)$, and outputs information about the homogeneous elements of the polynomial algebra $\C[V]$ which are fixed by all the transformations in $\Gamma$. In other words, classical invariant theory calculates the set of homogeneous $F \in \C[V]$ such that for all $\gamma \in \Gamma$, we have
\begin{equation*}
    F \circ \gamma = F.
\end{equation*}
The set of such $F$ forms a ring, called the \emph{ring of polynomial invariants}, and is denoted $\C[V]^{\Gamma}$. A basic method used in classical invariant theory to furnish polynomial invariants is to use the Reynolds operator, which is the projection $\C[V] \to \C[V]^{\Gamma}$ defined by
    \begin{equation*}
        F \mapsto \frac{1}{|\Gamma|} \sum_{\gamma \in \Gamma} (F \circ \gamma).
    \end{equation*}
The first interesting example takes $V = \C^2$ and $\Gamma$ to be the representation of $C_2$, which maps the non-identity element to $\left[\begin{smallmatrix} 0 & 1 \\ 1 & 0 \end{smallmatrix}\right]$.
Then $\C[V]^\Gamma$ is the ring of homogeneous symmetric polynomials in two variables.

Klein found, and the reader may directly check, that given a homogeneous polynomial in two variables invariant under the action of $\Gamma$, i.e., $F \in \C[V]^\Gamma$, the map $f : \PP^1 \to \PP^1$ defined in coordinates by
\begin{equation*}
    [x : y] \mapsto \left[ -\frac{\partial F}{\partial y} : \frac{\partial F}{\partial x} \right]
\end{equation*}
satisfies $\Gamma \subseteq \Aut(f)$. Doyle and McMullen derived another more general construction, again using classical invariant theory, which creates maps with automorphism group containing $\Gamma$ \cite{DM}. Specifically, given two invariants $F,G$ with degrees satisfying $\deg(F)= \deg(G)+2$ (or $G=0$), the map is given by
\begin{equation}\label{eq_DM}
    [x : y] \mapsto \left[ \frac{xG}{2}+\frac{\partial F}{\partial y} : \frac{yG}{2} - \frac{\partial F}{\partial x} \right].
\end{equation}
They also prove analytically that every dynamical system with automorphism group containing $\Gamma$ arises from their construction. With this machine for creating dynamical systems with symmetries, the only concern is that we might not exactly have $\Gamma = \Aut(f)$. To be sure we have a solution to the realizability problem, we must check against the existence of \emph{extra automorphisms.}  De Faria and Hutz used in \cite{Hutz15} this machinery to solve the realizability problem over $\C$ as well as to produce infinite families where every member of the family has automorphism group containing $\Gamma$.

Now we replace the base field $\C$ by $\bar{\FF}_p$ and explain how the above story morphs at each step.
\begin{itemize}
    \item As shown by the classification of Faber \cite{Faber}, there are many more conjugacy classes to test.

    \item If $\Gamma$ is $p$-regular, the same formula for the Reynolds operator works, and much of the classical theory over $\C$ carries over with minor modification. But if $\Gamma$ is $p$-irregular, the Reynolds operator is unavailable, and it can be computationally more difficult to locate polynomial invariants. This suggests the basic dichotomy present in modern commutative algebra between \emph{modular invariant theory} (the case where $p$ divides $|\Gamma|$) and its complement \emph{nonmodular invariant theory}. For an excellent reference that emphasizes this dichotomy, see \cite{smith}. Our investigation opens a new field of application for modular invariant theory. In particular, any work on the realizability problem in higher dimensions will probably require a deeper description of modular invariants than is presently available.

    \item The Klein and Doyle-McMullen constructions, which are the bridge from invariant theory to dynamics, may fail for various reasons in characteristic $p$. For instance, if we attempt the Doyle-McMullen construction with $F(x,y) = 0, G(x,y) = x^p + y^p$, we obtain the nonsense map $[0:0]$. Evidently, some constraints on degree are necessary. Even so, if the construction actually produces a valid map of degree at least 2, then it is easy to check that $\Gamma \subseteq \Aut (f)$.

    The converse---that all maps with $\Gamma \subseteq \Aut (f)$ arise from the Doyle-McMullen construction---is much harder to see, and some subtleties particular to positive characteristic arise. We build up the theory of this correspondence in Section \ref{sect_dm}, with our analogue of the Doyle-McMullen correspondence presented as Theorem \ref{thm_dmcharp}.

    \item Over $\C$, the central task is writing down an example $f$ such that $\Gamma \subseteq \Aut(f)$, and the problem of extra automorphisms has been addressed for a few special cases. Over $\bar{\FF}_p$, the problem of extra automorphisms is in some sense the whole point. We will see that the automorphism group of $f(z) = z^q$ is $\PGL_2(\FF_q)$, and every finite subgroup $\Gamma$ of $\PGL_2(\bar{\FF}_p)$ is contained in $\PGL_2(\FF_q)$ for a large enough choice of $q$. For each prime power $q$, this gives us a single example $f$  such that $\Gamma \subseteq \Aut(f)$ for every finite subgroup $\Gamma$ of $\PGL_2(\FF_q)$. So, the difficulty arises in how to create maps $f$ with some prescribed symmetries without picking up lots of others.
\end{itemize}

The following essential proposition, due to Faber-Manes-Viray, links the existence of automorphisms of certain order to the degree of the map.
\begin{prop}[\cite{FMV}, Proof of Proposition~2.4] \label{prop_degtest}
    Let $p$ be a prime, let $n \in \N$, and let $f: \PP^1 \to \PP^1$ be a rational map over $\bar{\FF}_p$ admitting an automorphism of order $n$. Then
\begin{equation}\label{degtest}
    \deg(f) \equiv -1, 0, 1 \mod{n}.
\end{equation}
If $n = p$, then up to conjugation, we further have $f(z) = \psi(z^p-z)+z$ for some rational map $\psi$, and
\begin{equation}\label{degtest_primepower}
    \deg(f) \equiv 0, 1 \mod{p}.
\end{equation}
\end{prop}


\subsection{Realizability of $p$-irregular subgroups of $\PGL_2(\FF_q)$} \label{sect_pirregular}

\subsubsection{Realizing $\PGL_2(\FF_q)$}
We now show that, for any power $q$ of a prime $p$, the group $\PGL_2(\FF_q)$ is realizable over $\FF_p$.
\begin{proof}[Proof of Theorem \ref{thm_realizable_p_irregular} part \ref{thm_realizable_pgl}]
    Let $f(z) = z^q$. For any rational map $g \in \bar{\FF}_q(z)$, we have $g(z^q) = g(z)^q$ if and only if $g$ is defined over $\FF_q$. Restricting $g$ to be degree 1, we find that $\Aut(g) = \PGL_2(\FF_q)$. We further claim that any map with degree at least $2$ and with automorphism group $\PGL_2(\FF_q)$ has degree at least $q$. To see this, let $g$ be a such a map. Let $a \in \FF_q$ be an affine fixed point of $g$; such a point exists since any rational map of degree at least 2 has at least 2 fixed points. For every $b \in \FF_q$, we have an automorphism $z \mapsto z + b$ of $g$, so every point of the form $a + b$ where $b \in \FF_q$ is fixed by $g$. Since $g$ has at least $q$ fixed points, we deduce that $\deg g \geq q - 1$. Further, since $g$ has an automorphism of order $p$, we know that $\deg g \equiv 0$ or $\deg g \equiv 1 \pmod p$, by \eqref{degtest_primepower}. We conclude that $\deg g \geq q$.
\end{proof}

Somewhat surprisingly, for any prime power $q$, the map $\frac{1}{z^q}$ also has an exact automorphism group isomorphic to $\PGLq$, because $z^q$ is conjugate to $\frac{1}{z^q}$. In fact, it is a quadratic twist.
    \begin{prop}
        For any prime power $q$, let $\zeta_{q+1}$ be a primitive $(q+1)$-th root of unity, and let
        $\tau =\begin{pmatrix} 1 & \zeta_{q+1} \\ \zeta_{q+1} & 1 \end{pmatrix}$.
        Then $\tau\in\PGL_2(\FF_{q^2})$ and conjugation by $\tau$ maps $f(z)=z^q$ to $f^{\tau}(z)=\frac{1}{z^q}$.
    \end{prop}
    \begin{proof}
        Checking the conjugation is a simple calculation, and $\zeta_{q+1}$ is in a quadratic extension of $\FF_q$ because $\FF_{q^2}^{\ast}$ is cyclic of order $q^2-1=(q-1)(q+1)$.
    \end{proof}
    It turns out that there are many elements of $\PGL_2(\FF_{p^2})$ that conjugate $z^q$ to $\frac{1}{z^q}$. This can be explained by the following result.
    \begin{prop}
        Let $f,g \in \Rat_d$ be in the same conjugacy class. Then the set $\Conj(f,g)$ of all conjugations from $f$ to $g$ is a right coset of $\Aut(f)$.
    \end{prop}
    \begin{proof}
        Let $\tau\in \Conj(f,g)$. We must show that $\Aut(f)\circ \tau=\Conj(f,g)$. For the first containment, let $\beta\in\Aut(f)$. Then $f^{\beta \circ \tau}=(f^{\beta})^{\tau}=f^{\tau}=g$, and so $\beta\circ \tau\in\Conj(f,g)$.

        For the reverse containment, it suffices to show that for all $\tau,\beta\in\Conj(f,g)$, we have $\tau\circ \beta^{-1}\in\Aut(f)$. This holds, since $f^{\tau\circ \beta^{-1}}=(f^\tau)^{\beta^{-1}}=g^{\beta^{-1}}=f$.
    \end{proof}

    \begin{proof}[Proof of Theorem \ref{thm_realizable_p_irregular} part \eqref{thm_realizable_semielementary}]
        Let $\Gamma$ be a $p$-semi-elementary subgroup of $\PGL_2(\FF_q)$. For the purposes of the realizability problem, we can replace $\Gamma$ by a conjugate. Then by the classification of Faber \cite{Faber}, as presented in Proposition \ref{prop_faber}, we can assume that $\Gamma$ has the following form:
        \begin{itemize}
            \item The group $\Gamma$ is a subgroup of the Borel group; that is, all its elements are of the form $z \mapsto az + b$.
            \item For any integer $n \geq 1$, let $\mu_n$ denote the multiplicative group of $n$-th roots of unity in $\Fpbar$. There is an additive group $\Lambda \subseteq \FF_q$ and an integer $n \geq 1$ such that
            $$\Gamma = \{ z \mapsto az + b : a \in \mu_n, b \in \Lambda \}.$$
            \item Multiplication by elements of $\mu_n$ maps $\Lambda$ into $\Lambda$.
        \end{itemize}

        Let
        \begin{equation*}
            f(z) = \prod_{\lambda \in \Lambda} (z - \lambda)^{n + 1} + z.
        \end{equation*}
        Then we claim $\Aut(f) = \Gamma$.
        Say $\tau \in \Gamma$. Then $\tau$ is given by $\tau(z) = az + b$ for some $a, b$ where $b \in \Lambda$ and $a \in \mu_n$. The following sequence of equalities is justified by re-indexing the product twice.
        \begin{align*}
        f(\tau(z)) &= \prod_{\lambda \in \Lambda} (az + b - \lambda)^{n+1} + az + b \\
        &= \prod_{\lambda \in \Lambda} (az - \lambda)^{n + 1} + az + b \\
        &= a^{n+1} \prod_{\lambda \in \Lambda} (z - \lambda/a)^{n+1} + az + b \\
        &= a \prod_{\lambda \in \Lambda} (z - \lambda/a)^{n + 1} + az + b \\
        &= a \prod_{\lambda \in \Lambda} (z - \lambda)^{n+1} + az + b \\
        &= \tau(f(z)).
        \end{align*}
    So, $\Gamma \subseteq \Aut (f)$. Now we prove the reverse containment. Suppose that $\tau \in \Aut(f)$. The fixed points of $f$ are $\Lambda \cup \{ \infty \}$. The multiplier at $\infty$ is 0 because $\infty$ is a critical point, and the multiplier at each point of $\Lambda$ is 1 (since $n \geq 1$). Then $\tau$ must fix $\infty$, since it is the only fixed point of $f$ with multiplier 0, so $\tau$ is of the form $z \mapsto a z + b$. Now, we must show that $b$ is in $\Lambda$ and that $a$ is in $\mu_n$. We consider the equality of polynomials given by $f(\tau(z)) = \tau(f(z))$:
    \begin{equation*}
        f(az + b) = af(z) + b.
    \end{equation*}
    The leading coefficient on the left side is $a^{n+1}$, and the leading coefficient on the right is $a$, so $a$ is an $n$-th root of unity. Expanding the equality of polynomials,
    \begin{equation*}
        a \prod_{\lambda \in \Lambda} (z + b/a - \lambda/a)^{n+1} + az + b = a \prod_{\lambda \in \Lambda} (z - \lambda)^{n+1} + az + b.
    \end{equation*}
    Simplifying, we have
    \begin{equation*}
        \prod_{\lambda \in \Lambda} (z + b/a - \lambda/a)^{n+1} = \prod_{\lambda \in \Lambda} (z - \lambda)^{n+1}.
    \end{equation*}
    Then $z \mapsto az + b$ must map $\Lambda$ to $\Lambda$ bijectively. The map $z \mapsto z/a$ is also bijective, so composing, we find that $z \mapsto z + b$ maps $\Lambda$ to $\Lambda$. Therefore, $b \in \Lambda$, completing the proof.
    \end{proof}

    \subsubsection{Realizing $\PSL_2(\FF_q)$}
    We now show how to realize $\PSL_2(\FF_q)$, where $q$ is a power of an odd prime $p$. We assume $p > 2$ to ensure that $\PSL_2(\FF_q) \neq \PGL_2(\FF_q)$.

    \begin{proof}[Proof of Theorem \ref{thm_realizable_p_irregular} part \eqref{thm_realizable_psl}] We begin with the fundamental invariants of $\PSLq$,
        \begin{align*}
            u &= x^qy - xy^q,\\
            c_1 &= \sum_{n = 0}^{q} x^{(q-1)(q-n)}y^{(q-1)n},
        \end{align*}
        which have degree $q+1$ and $q^2-q$ respectively (see, for instance, \cite{Benson}).

        The Doyle-McMullen construction \cite{DM} takes two invariant homogeneous polynomials $F$ and $G$ of some $\Gamma \subseteq \PGL_2$ and outputs a map with $\Gamma \subseteq \Aut(f)$. We generalize this construction to characteristic $p>0$ in Theorem \ref{thm_dmcharp}. For invariants $F$ and $G$, the corresponding map on projective space is
        $f = [xF + G_y : yF - G_x]$, where $G_y$ and $G_x$ are the partial derivatives.

        Applying this construction to $G=u^a$ and $F=c_1^b$, with $a$ and $b$ as given in the statement of the theorem and using that $a \equiv 1 \mod p$, we obtain the map
        \begin{align*}
            f(x,y)  = \Bigg[& x \left(\sum_{n = 0}^{q} x^{(q-1)(q-n)}y^{(q-1)n}\right)^b + (x^qy - xy^q)^{a-1}x^q : \\
            & y \left(\sum_{n = 0}^{q} x^{(q-1)(q-n)}y^{(q-1)n}\right)^b + (x^qy - xy^q)^{a-1}y^q\Bigg].
        \end{align*}

        Next, we calculate the fixed points of $f$:
        \begin{align*}
            f(x,y) &= [x : y]\\
            \iff \quad
            &y\left(x\left(\sum_{n = 0}^{q} x^{(q-1)(q-n)}y^{(q-1)n}\right)^b + (x^qy - xy^q)^{a-1}x^q\right)\\
            &=x\left(y\left(\sum_{n = 0}^{q} x^{(q-1)(q-n)}y^{(q-1)n}\right)^b + (x^qy - xy^q)^{a-1}y^q\right)\\
            \iff \quad  &(x^qy - xy^q)^{a-1}x^q y = (x^qy - xy^q)^{a-1}x y^q\\
            \iff \quad &(x^qy - xy^q)^{a-1}(x^qy - xy^q) = (x^qy - xy^q)^{a} = 0.
        \end{align*}

        Setting $y = 1$, we see that the fixed points are the roots of $(x^q - x)^a$, or the elements of $\FF_q$, each with multiplicity $a$. Likewise, $y = 0$ is a solution, so infinity is a fixed point with multiplicity $a$.

        We know that $\PSLq \subseteq \Aut(f)$ by construction. It remains to show equality. Using the assumption $p > 2$, let $\alpha$ be any non-square element of $\FF_q$. Then $\left(\begin{smallmatrix}\alpha & 0 \\0 & 1 \end{smallmatrix}\right)$ corresponds to the map $\tau(x,y) = [\alpha x : y]$. We claim that $\tau \not\in \Aut(f).$ Indeed, we compute
        \begin{align*}
            f^{\tau} = \Bigg[&\frac{1}{\alpha}\left( \alpha x \left(\sum_{n = 0}^{q} (\alpha x)^{(q-1)(q-n)}y^{(q-1)n}\right)^b + ((\alpha x)^qy - (\alpha x)y^q)^{a-1}((\alpha x)^q)\right) : \\
            & y\left(\sum_{n = 0}^{q} (\alpha x)^{(q-1)(q-n)}y^{(q-1)n})\right)^b + ((\alpha x)^qy - (\alpha x)y^q)^{a-1}y^q \Bigg]\\
            &= \Bigg[ x \left(\sum_{n = 0}^{q}x^{(q-1)(q-n)}y^{(q-1)n}\right)^b + \alpha^{a-1}(x^qy - xy^q)^{a-1}x^q :\\
            & y \left(\sum_{n = 0}^{q}x^{(q-1)(q-n)}y^{(q-1)n}\right)^b + \alpha^{a-1}(x^qy - xy^q)^{a-1}y^q \Bigg].
        \end{align*}
        Thus we see that $f^{\tau} = f \iff \alpha^{a-1} = 1$.
        Since $a = \frac{q(q-3) + 4}{2}$,
        \begin{align*}
            a-1 &= \frac{q(q-3) + 2}{2} = \frac{(q-1)(q-2)}{2}\\
                &\implies \alpha^{a-1} = (\alpha^{\frac{q-1}{2}})^{q-2} = (-1)^{q-2} = -1.
        \end{align*}
        Therefore, $f^{\tau} \neq f$.

        Since $f$ has $q + 1$ fixed points, we have $\left|\Aut(f)\right| \leq (q+1)q(q-1)$. Further, since $z \mapsto \alpha z$ induces a self-map of $\Fix(f)$ but is not an automorphism of $f$, the inequality is strict. Since $\Aut(f)$ contains $\PSL_2(\FF_q)$, we have
        \begin{equation} \label{eq_psl_order_estimate}
            \frac{(q+1)q(q-1)}{2} \leq \left| \Aut(f) \right| < (q+1)q(q-1).
        \end{equation}
        Since $\Aut(f)$ contains $\PSL_2(\FF_q)$, it is a $p$-irregular group that does not stabilize any subset of $\PP^1(\bar{\FF}_q)$ of cardinality 1 or 2. By the classification of finite subgroups (see the Appendix), if $p \neq 3$, the only such groups are isomorphic to $\PGL_2(\FF_{q'})$ or $\PSL_2(\FF_{q'})$ for some power $q' \geq q$ of $p$, and among these, the only isomorphism type satisfying \eqref{eq_psl_order_estimate} is $\PSL_2(\FF_{q})$. If $p = 3$, the same argument applies, but we must also rule out the isomorphism type $A_5$. But the order of $A_5$ is 60, which does not satisfy \eqref{eq_psl_order_estimate} for any power $q$ of 3.

        The final claims to verify are that
        \begin{equation} \label{eq_deg_psl_example}
            \deg f = \frac{1}{2}(q^3 - 2q^2 + q + 2),
        \end{equation}
        and that this $f$ is of minimal degree for $\PSL_2(\FF_q)$. In Theorem \ref{thm_pslminimal}, we show that any map $f$ with $\deg f > 1$ and $\Aut(f) \cong \PSL_2(\FF_q)$ has degree at least $\frac{1}{2}(q^3 - 2q^2 + q + 2).$ The formula for $f$ shows that $\deg f \leq \frac{1}{2}(q^3 - 2q^2 + q + 2)$. The set of fixed points has cardinality $q + 1$, so $\deg f > 1$, proving \eqref{eq_deg_psl_example}.
    \end{proof}

    Theorem \ref{thm_realizable_p_irregular} part \eqref{thm_realizable_psl} is rather cumbersome and it is difficult to understand the maps arising from the invariants. In the case where $q=p$, we have the following simplified version.
    \begin{thm}
        Let $p>2$ be prime, let $m = \frac{1}{2}p^2-\frac{3}{2}p+2$, and let $c \neq 0$, and let
        \begin{equation*}
            \psi(z)=\frac{cz^m}{(z^{p-1}+1)^{\frac{p-1}{2}}+cz^{m-1}}.
        \end{equation*}
        Then the automorphism group of $f(z)=\psi(z^p-z)+z$ is exactly $\PSLp$.
    \end{thm}
    \begin{proof}
        We check that the generators of $\PSLp$, which are
        $\begin{pmatrix}
            1 & 1 \\
            0 & 1
        \end{pmatrix}$,
        $\begin{pmatrix}
            0 & 1 \\
            -1 & 0
        \end{pmatrix}$,
        $\begin{pmatrix}
            \alpha & 0 \\
            0 & 1
        \end{pmatrix}$, where $\alpha$ is a quadratic residue in $\FF_p$, are all automorphisms of $f$. We also need to show that
        $\begin{pmatrix}
            \alpha & 0 \\
            0 & 1
        \end{pmatrix}$ is not an automorphism of $f$ when $\alpha$ is a non-residue.

        For
        $\begin{pmatrix}
            1 & 1 \\
            0 & 1
        \end{pmatrix}$, we compute
        \begin{equation*}
            f(z+1) - 1 = \psi((z+1)^p - (z+1)) + (z+1) -1 = \psi(z^p-z) + z = f.
        \end{equation*}

        We next check maps of the form
        $\begin{pmatrix}
            \alpha & 0 \\
            0 & 1
        \end{pmatrix}$. This is an automorphism if and only if $f(\alpha z) = \alpha f(z)$. This holds if and only if
        \begin{equation}\label{eq1}
            0 = f(\alpha z) - \alpha f(z) = \alpha\psi(z^p - z) - \psi(\alpha(z^p - z)).
        \end{equation}
        Making the substitution $w = z^p - z$, we see equation \eqref{eq1} holds if and only if
        \begin{align*}
            0 &= \alpha\psi(w) - \psi(\alpha w) =
            \frac{\alpha c w^m}{(w^{p-1} + 1)^\frac{p-1}{2} + cw^{m-1}} -
            \frac{c\alpha^m  w^m}{(\alpha^{p-1}w^{p-1} + 1)^\frac{p-1}{2} + c\alpha^{m-1}w^{m-1}}\\
            &= (c\alpha w^m)\left(\frac{1}{(w^{p-1} + 1)^\frac{p-1}{2} + cw^{m-1}} -
            \frac{\alpha^{m-1}}{(\alpha^{p-1}w^{p-1} + 1)^\frac{p-1}{2} + c\alpha^{m-1}w^{m-1}}\right).
        \end{align*}
        Keeping in mind that $\alpha^{p-1} = 1$, this is equivalent to
        \begin{align*}
            0 &= (\alpha^{p-1}w^{p-1} + 1)^\frac{p-1}{2} + c\alpha^{m-1}w^{m-1} -
            \alpha^{m-1}(w^{p-1} + 1)^\frac{p-1}{2} - \alpha^{m-1}cw^{m-1}\\
            &= (1 - \alpha^{m-1})(w^{p-1} + 1)^{\frac{p-1}{2}}.
        \end{align*}
        Thus,
        $\begin{pmatrix}
            \alpha & 0 \\
            0 & 1
        \end{pmatrix}$
        is an automorphism of $f$ if and only if $\alpha^{m-1} = 1$. We have
        \begin{equation*}
            \alpha^{m-1} = \alpha^\frac{(p-1)(p-2)}{2}
        \end{equation*}
        and the order of $\alpha$ is $\FF_p$ must be a divisor of $p-1$. So $\alpha^{m-1} = 1$ if and only if $\alpha^\frac{p-1}{2} = 1$, which, by Euler's criterion, is equivalent to $\alpha$ being a quadratic residue.

        It remains to check that
        $\begin{pmatrix}
            0 & 1 \\
            -1 & 0
        \end{pmatrix}$
        is an automorphism. To simplify the computations, we introduce the variables $x = z^p - z$, $y = z^{p+1}$.
        We need $\frac{-1}{f(z)} - f(\frac{-1}{z}) = 0$.
        We have
        \begin{equation*}
            \frac{-1}{f(z)} - f\left(\frac{-1}{z}\right) =
            \frac{-1}{\psi(z^p-z) + z} - \psi\left(-\frac{1}{z^p} + \frac{1}{z}\right) + \frac{1}{z} = \frac{-1}{\psi(x) + z} - \psi\left(\frac{x}{y}\right) + \frac{1}{z},
        \end{equation*}
        which vanishes if and only if
        \begin{equation*}
           z(\psi(x) + z)\psi\left(\frac{x}{y}\right) - \psi(x)=0.
        \end{equation*}
        Now
        \begin{align*}
            z(\psi(x) &+ z)\psi\left(\frac{x}{y}\right) - \psi(x)\\
            &=z\left(
            \frac{cx^m}{(x^{p-1} + 1)^\frac{p-1}{2} + cx^{m-1}} + z
            \right)\left(
            \frac{c(\frac{x}{y})^m}{((\frac{x}{y})^{p-1} + 1)^\frac{p-1}{2} + c(\frac{x}{y})^{m-1}}
            \right) - \frac{cx^m}{(x^{p-1} + 1)^\frac{p-1}{2} + cx^{m-1}}\\
            &=z\left(
            \frac{cx^m}{(x^{p-1} + 1)^\frac{p-1}{2} + cx^{m-1}} + z
            \right)\left(
            \frac{cx^m}{y^m((\frac{x}{y})^{p-1} + 1)^\frac{p-1}{2} + cyx^{m-1}}
            \right) - \frac{cx^m}{(x^{p-1} + 1)^\frac{p-1}{2} + cx^{m-1}},
        \end{align*}
        which vanishes precisely with
        \begin{align}
            cx^m&\left[cx^mz + z^2\left((x^{p-1} + 1)^\frac{p-1}{2} + cx^{m-1}
            \right) - y^m\left(\left(\frac{x}{y}\right)^{p-1} + 1\right)^\frac{p-1}{2} - cyx^{m-1}\right] \notag\\
            &= cx^m\left[cx^{m-1}(zx+z^2-y) + z^2\left((x^{p-1} + 1)^\frac{p-1}{2}\right) - y^\frac{-p + 3}{2}(x^{p-1} + y^{p-1})^\frac{p-1}{2}\right]. \label{eq2}
        \end{align}
        Now we use the following two identities:
        \begin{align*}
            y &= zx + z^2\\
            x^{p-1} + y^{p-1} &= z^{p-1}(x^{p-1} + 1).
        \end{align*}
        The first identity is trivial, and the second follows from the expansion
        \begin{equation*}
            x^{p-1} \equiv z^{p(p-1)} + z^{(p-1)(p-1)} + z^{(p-2)(p-1)} + z^{(p-3)(p-1)} + \cdots + z^{p-1} \pmod{p},
        \end{equation*}
        using that $\binom{p-1}{k} \equiv (-1)^k \pmod{p}$ for $1 \leq k \leq p-1$.
        With these identities, \eqref{eq2} becomes
        \begin{align*}
            cx^m\Big[&z^2\left((x^{p-1} + 1)^\frac{p-1}{2}\right) - y^{\frac{-p + 3}{2}}z^{\frac{(p-1)^2}{2}}(x^{p-1} + 1)^{\frac{p-1}{2}}\Big]\\
            &=cx^m(x^{p-1} + 1)^{\frac{p-1}{2}}\left[z^2 - z^{\frac{(p+1)(-p + 3)}{2} + \frac{(p-1)^2}{2}}\right]\\
            &=cx^m(x^{p-1} + 1)^{\frac{p-1}{2}}\left[z^2 - z^2\right]=0.
        \end{align*}
        Thus,
        $\begin{pmatrix}
            0 & 1 \\
            -1 & 0
        \end{pmatrix}$
        is indeed an automorphism.

        We have so far shown
        $$\PSL_2(\FF_p) \subseteq \Aut(f) \subsetneq \PGL_2(\FF_q).$$
        To show that $\Aut(f)$ is not equal to any group strictly containing $\PSL_2(\FF_p)$, we argue as in the proof of Theorem \ref{thm_realizable_p_irregular} part \ref{thm_realizable_psl}. The map $f$ has $p + 1$ fixed points, as can be seen from the equation
        $$f(z) = \psi(z^p - z) + z = z.$$
        There are $(p+1)p(p-1)$ elements of $\PGL_2(\bar{\FF}_p)$ for which $\Fix(f)$ is an invariant set, and we showed that at least one of these elements is not an automorphism of $f$. The argument about group orders in the proof of Theorem \ref{thm_realizable_p_irregular} part \ref{thm_realizable_psl} then shows that $\Aut(f) = \PSL_2(\FF_p)$.
\end{proof}

From the earlier discussion of maps of the form $\psi(z^p - z) + z$, we can easily determine the multiplier spectrum of the above map, which shows that varying $c$ results in a $1$-dimensional family of maps realizing $\PSLp$ in the moduli space.

\subsubsection{$p$-irregular dihedral groups}
We now prove that any $p$-irregular dihedral group $D_{2n}$ is realizable. These groups occur only when $p = 2$ and $n \geq 3$ is odd.
\begin{proof}[Proof of Theorem \ref{thm_realizable_p_irregular} part \ref{thm_realizable_irreg_dih_2}]
Let
$$ f(z) = \frac{1}{z^{2n-1}}.$$
We claim that $\Aut(f) \cong D_{2n}$. By direct computation, the automorphism group of $f$ includes $z \mapsto 1/z$ and $z \mapsto \zeta_n z$, and these transformations generate a dihedral group of order $2n$. We now show that there are no extra automorphisms by showing that $f(z)$ has at most $2n$ automorphisms. First, notice that
$$\Fix(f) = \{z \in \bar{\FF}_2 : z^{2n} = 1\},$$
which has cardinality $n$. Second, we claim that the set of points with a unique preimage is $\{0, \infty\}$. It is clear that $f^{-1}(0) = \infty$ and $f^{-1}(\infty) = 0$. Now, let $c \neq 0, \infty.$ The affine preimages of $c$ are the values of $z$ such that $f(z) = c$, or equivalently
$$c z^{2n-1} = 1.$$
Since $c \neq 0$, this polynomial is of degree $2n - 1$, and is separable because $2n - 1$ is odd, so the roots are distinct. Since $n > 1$, there are multiple roots.
Each automorphism of $f$ is determined by where it sends $0$, $1$, and $\infty$. There are at most 2 possible images for $0$, and $n$ possible images for $1$, so there are at most $2n$ automorphisms of $f$.
\end{proof}

\subsubsection{The icosahedral subgroup of $\PGL_2(\bar{\FF}_3)$.} \label{sect_irreg_a5}
The final construction needed for Theorem \ref{thm_realizable_p_irregular} is the case of the icosahedral subgroup $A_5$ in $\PGL_2(\bar{\FF}_3)$. This is a finite calculation that is part of the dynamical systems library in Sage \cite{sage}.
\begin{proof}[Proof of Theorem \ref{thm_realizable_p_irregular} part (\ref{thm_realizable_irreg_a5})]
Let $f$ be the dynamical system in the theorem statement. Using the algorithm in Sage to compute automorphism groups of dynamical systems, we computed that $\Aut(f) \cong A_5$ \cite{sage}. We also compute in Sage that the resultant of $f$ is nonzero, so $\deg f = 21.$ We prove the claim that $f$ is of minimal degree for $A_5$ in Theorem \ref{thm_a5_minimal}.
\end{proof}

\subsection{Realizability of $p$-regular finite subgroups of $\PGLpbar$} \label{sect_pregular}

In this section, we construct solutions to the realizability problem for every $p$-regular finite subgroup $\Gamma$ of $\PGLpbar$. Each group that appears is trivial, cyclic, dihedral, tetrahedral, octahedral, or icosahedral, and each isomorphism type appears as a single conjugacy class; see the Appendix for the classification.

\subsection{The trivial group}
We now prove Theorem \ref{thm_realizable_p_regular} part (\ref{p_regular_trivial}), which says that the trivial group $1$ is realizable.
\begin{proof}[Proof of Theorem \ref{thm_realizable_p_regular} part (\ref{p_regular_trivial})]
    Let $f(z) = z^2 + z$. By direct calculation, we have $\Fix(f) = \{0,\infty\}$. The multiplier at $1$ is $1$, and the multiplier at $\infty$ is $0$. Thus any automorphism of $f$ must fix $0$ and $\infty$, so the only possible automorphisms are of the form $z \mapsto \alpha z$, where $\alpha \neq 0$. Given such an automorphism, $f(\alpha z) = \alpha f(z)$ implies $\alpha^2 = \alpha$, so $\alpha = 1$. Thus $\Aut(f) = 1.$
\end{proof}
\begin{rmk}
    This is a special case of the argument used to realize $p$-semi-elementary groups in Theorem \ref{thm_realizable_p_irregular} part (\ref{thm_realizable_semielementary}), taking the multiplicative group $\mu_n = 1$ and the additive group $\Lambda = 0$. In Section \ref{sect_Ad}, we further show that a generic degree-2 map $f$ has no nontrivial automorphisms.
\end{rmk}

\subsection{Cyclic groups}
    In this section, we prove Theorem \ref{thm_realizable_p_regular} part (\ref{p_regular_cyclic}), that every $p$-regular cyclic group $C_{n}$ arises as the exact automorphism group of a self-map of $\PP^1(\Fpbar)$. The $p$-regularity condition means that $n$ is coprime with $p$.

    Silverman \cite{Silverman12} shows that, in characteristic $0$, a map $f$ has $C_n \subseteq \Aut(f)$ if and only if $f$ is of the form $f(z) =z \psi(z^n)$ for some rational function $\psi$. The argument is valid as long as primitive $n$-th roots of unity exist, which is true in characteristic $p$ when $\gcd(p,n)=1$.

    \begin{proof}[Proof of Theorem \ref{thm_realizable_p_regular} part \ref{p_regular_cyclic}]
         Let $n$ be coprime with $p$, and let $f(z)=\frac{1}{z^{n-1}}+z$.

        First notice that the map $z\mapsto \zeta_n z$ is an order $n$ automorphism. For the other containment, notice that $\infty$ is the unique fixed point of $f$. The unique non-fixed preimage of $\infty$ is $0$. Any automorphism of $f$ therefore must fix $\infty$ and $0$, and so is of the form $\alpha(z) = az$ for some constant $a$. We compute $f^{\alpha} = \frac{1}{a^nz^{n-1}} + z$, so to get an automorphism, we must have that $a$ is a primitive $n$-th root of unity.

        It remains to show that no map of smaller degree has $C_n$ as its automorphism group. By Silverman \cite{Silverman12}, if a map $f$ has an order $n$ automorphism with $n$ co-prime to $p$, it must be of the form $z\psi(z^n)$ for some rational map $\psi$. If $\psi$ is a constant map, then $f$ is degree 1; otherwise, the minimal degree possible is $n-1$ when $\psi(z) = \frac{a}{z}$ with $a\neq 0$. In this case $f(z)=\frac{a}{z^{n-1}}$ has the extra automorphism $z\mapsto\frac{1}{z}$. Thus, there are no maps of degree $n-1$ with $C_n$ as their exact automorphism group, and $n$ is the minimal degree.
    \end{proof}

    \begin{rmk}
         Let $p$ be a prime and let $n \geq 2$ be coprime to $p$. Then the map $f(z)=z^{n+1}+z$ also has $\Aut(f)\cong C_n$. This $f(z)$ appears when applying the construction used to prove Theorem \ref{thm_realizable_p_irregular} part  (\ref{thm_realizable_semielementary}) to the multiplicative group $\mu_n$ of $n$-th roots of unity and the additive group $\Lambda = 0.$ However, $f(z)$ is not of minimal degree for $C_n$.
    \end{rmk}

\subsection{Dihedral groups}
    In this section, we prove Theorem \ref{thm_realizable_p_regular} part \eqref{p_regular_dihedral} that every $p$-regular dihedral group $D_{2n}$ arises as the exact automorphism group of a self-map of $\PP^1(\Fpbar)$. The $p$-regularity condition here means that $p > 2$ and that $n$ is coprime to $p$.

    Silverman described maps with automorphism group containing a dihedral group $D_{2n}$; see \cite{Silverman12} or \cite[Exercise 4.37]{Silverman10}. In characteristic 0, these are exactly the maps of the form
    \begin{equation*}
        f(z) = z \cdot \frac{F(z^n)}{z^{dn} F(z^{-n})},
    \end{equation*}
    where $F$ is any polynomial and $d$ is its degree. Using the form above, one can write down various families of maps with at least dihedral symmetry and then check against extra automorphisms. For instance, in characteristic 0, the realizability problem for $D_{2n}$ can be solved by $z^{n + 1}$, which corresponds to the choice $F(z) = z$. But in characteristic $p$, this $f$ sometimes acquires extra automorphisms. The task for us is to find families of solutions that each work for most choices of $p$ and $n$, so that taken together, all choices of $p$ and $n$ are accounted for.

\begin{proof}[Proof of Theorem \ref{thm_realizable_p_regular} part \eqref{p_regular_dihedral}]
    In each case to be addressed, the maps $\alpha(z)=1/z$ and $\beta(z) = \zeta_n z$ are automorphisms of $f$ that generate a dihedral group $D_{2n}$. To prove exactness, we argue that in each case, $f$ has at most $2n$ automorphisms.

    \begin{enumerate}
        \item We assume $n \not\equiv -1 \pmod{p}$ and $f(z) = z^{n+1}$. A simple calculation shows that $\alpha, \beta$ are automorphisms so that $D_{2n} \subseteq \Aut(f)$. Any automorphism must permute sets of fixed points of the same multiplier. Examining the equation $f(z) = z$, we calculate that the fixed points are $0$, $\infty$, and all the $n$-th roots of unity. Of these, the fixed points $0$ and $\infty$ have multiplier $0$, and the $n$-th roots of unity have multiplier $n + 1$, which is nonzero by the hypothesis on $n$. We conclude that every automorphism permutes $\{0, \infty\}$ and permutes $\{ {\zeta_n}^k: k = 0, 1, ..., n - 1 \}$.

        An automorphism can be completely described by specifying the images of three points. So we may bound the number of automorphisms by considering the possible images of 0, $\infty$, and 1. There are at most $n$ choices for where to send 1. There are at most 2 choices for where to send 0, and that choice also determines the image of $\infty$. So there are at most $2n$ automorphisms of $f$.

        \item We assume $n \not\equiv 1 \pmod{p}$, $n > 2$, and $f(z) = \frac{1}{z^{n-1}}$. A simple calculation checks that $\alpha, \beta$ are automorphisms, so $D_{2n} \subseteq \Aut(f)$. We  again prove that $\{{\zeta_n}^k : k = 0, ..., n - 1 \}$ and $\{0, \infty\}$ are invariant sets for every automorphism; then the argument in the first case proves the bound. The first set is $\Fix(f)$, so it is invariant. To prove invariance of $\{0, \infty\}$, we show that it is the set of all points with a unique preimage. A direct check shows that this set contains 0 and $\infty$, so we need only check that no other points are in the set. Suppose $c\not\in \{ 0,\infty \}$. Then the preimages of $c$ are the roots of $z^{n-1}-\frac{1}{c}$. Since $n \not\equiv 1 \mod p$ by hypothesis, this polynomial is separable, so it has distinct roots. Then, using the hypothesis $n > 2$, distinct roots implies at least two roots, so $c$ does not have a unique preimage.

        Since $D_{2n}$ has an element of order $n$, by Proposition \ref{prop_degtest}, the lowest degree map that could realize it is $n-1$, which is what we have.

        \item We assume $n=2$ and $f(z) = z \cdot \frac{z^2 + a}{a z^2 + 1}$ for $a$ not in the exceptional set $\{-3, -1, 0, 1 \}$. First we observe that $f$ has automorphisms $\alpha(z) = 1/z$ and $\beta(z)=-z$, which generate a dihedral group $D_4 \cong C_2 \times C_2$ since $p > 2$. The conditions on $a$ ensure that $f$ is degree 3. Then we just need to show that $f$ has at most 4 automorphisms. We can do this by finding two invariant sets of cardinality 2. We calculate the fixed points of $f$ and sort them by multiplier. The fixed points are 0, $\infty$, 1, and $-1$. The first two of these have multiplier $a$, and the last two have multiplier $(3 - a)/(a + 1)$. The conditions on $a$ guarantee that these two multipliers are distinct. Finally, this map $f(z)$ is of minimal degree for $D_4$, since by Theorem \ref{thm_Ad}, there are no degree 2 maps with automorphism group $D_4$.
    \end{enumerate}
    These cases account for all valid choices of $p$ and $n$.
\end{proof}

\subsection{Platonic solid groups}

    The problem of finding maps with platonic solid symmetry in characteristic 0 has been studied in detail by Klein \cite{klein_icosa}, Doyle and McMullen \cite{DM}, and De Faria and Hutz \cite{Hutz15}. Using their examples, we experimentally found that various self-maps of $\PP^1(\bar{\Q})$ with platonic solid symmetries could usually be reduced modulo $p$ to produce maps of $\PP^1(\bar{\FF}_p)$ without picking up extra automorphisms. We turn this observation into a proof of the remainder of Theorem \ref{thm_realizable_p_regular} by carrying out the following strategy.

\begin{enumerate}
    \item Exhibit a faithful representation of $\Gamma$ in $\PGL_2(\bar{\Q})$ where each entry of each matrix in $\Gamma$ is an algebraic integer. Then by reducing the entries of each matrix modulo $p$, we get entries in $\bar{\FF}_p$, and, in fact, we get a new representation of $\Gamma$ in $\PGL_2(\bar{\FF}_p)$. We denote the image of $\Gamma$ by $\Gamma_p$. We seek representations of $\Gamma$ such that the resulting representation in $\PGL_2(\bar{\FF}_p)$ is faithful for almost all $p$.
    \item Choose a map $f$ over $\bar{\Q}$ that has exact automorphism group $\Gamma$, and reduce it modulo $p$ to obtain a map $f_p$. The automorphism group of $f_p$ certainly contains $\Gamma_p$ but may have picked up additional elements as well.
    \item Show that for most primes, the reduced map $f_p$ has degree at least 2 and no automorphisms besides those in $\Gamma_p$.
    \item For any primes that have not been accounted for yet, make another choice of $f$ and repeat the process.
\end{enumerate}
As it turns out, most choices of $f$ seem to work for most primes $p$, so this strategy does not take long to terminate. The third step above is the most interesting, and our methods differ somewhat for the three platonic solid groups.

\begin{proof}[Proof of Theorem \ref{thm_realizable_p_regular} part \eqref{p_regular_octahedral}]
    The octahedral group $S_4$ has $24 = 2^3 \cdot 3$ elements, and we study only $p$-regular groups in this section, so this case only concerns primes $p > 3$.

    The octahedral group has a representation over $\bar{\Q}$ given by
    \begin{equation*}
        \Gamma = \left\langle
        S = \begin{bmatrix}
        i & i \\
        1 & -1 \\
        \end{bmatrix},
        T = \begin{bmatrix}
        i & 0 \\
        0 & 1 \\
        \end{bmatrix},
        U = \begin{bmatrix}
        0 & 1 \\
        1 & 0
        \end{bmatrix}\right\rangle,
    \end{equation*}
    where $i$ is a primitive fourth root of unity.
Now we check whether reduction is injective. When $p > 2$, the image of $i$ is still a primitive fourth root of unity. The subgroup $\Gamma'$ generated by $S, T^2$, and $U$ is tetrahedral. Reduction is injective on $\Gamma'$ since the elements $U, T^2, UT^2, U^2, S^2, S^3, US$ remain distinct, which means the image has cardinality at least 7. Then, the first isomorphism theorem of group theory shows that the homomorphism is injective. And reduction does not map $T$ into the image of $\Gamma'$, so the image of $\Gamma$ has at least 13 elements, so reduction is injective on $\Gamma$.

The paper by de Faria-Hutz \cite{Hutz15} gives examples of maps with exact automorphism group $\Gamma$. We first try reducing
\begin{equation*}
    f(z) = \frac{-z^5 + 5z}{5z^4 - 1}.
\end{equation*}
The resultant is $-2^{12} \cdot 3^4$, so the reduced map $f_p$ is degree 5 for all $p > 3$.

This gives us maps for every $p > 3$ with automorphism group containing $\Gamma_p$, but we need to check for extra automorphisms. Suppose $f$ has an extra automorphism. By the classification of subgroups (see the Appendix), all the finite subgroups of $\PGL_2(\bar{\FF}_p)$ strictly containing $S_4$ are $p$-irregular, so any extra automorphism implies the existence of an automorphism of order $p$. Then by Equation \eqref{degtest_primepower}, we know $\deg f \equiv 0, 1 \pmod p$, so $p \leq \deg f + 1= 6$. So we have exactly $\Gamma_p$ except possibly when $p = 5$. In that case, we need to try another $f$ since $f_5(z) = z^5$ has automorphism group $\PGL_2(\FF_5) \neq \Gamma_5$.

To account for the case $p = 5$, we try another choice,
\begin{equation*}
    f(z) = \frac{-7z^4-1}{z^7 + 7z^3}.
\end{equation*}
We compute the resultant $-2^{16} \cdot 3^4$ and find that 5 is not a factor, so the reduced map is degree 7. And the prime 5 passes the test of equation \eqref{degtest}, and, since we are working with a single prime, we compute directly in Sage \cite{sage} that the automorphism group is $S_4$.
\end{proof}

\begin{proof}[Proof of Theorem \ref{thm_realizable_p_regular} part \eqref{p_regular_icosahedral}]
The icosahedral group $A_5$ has $60 = 2^2 \cdot 3 \cdot 5$ elements, so this case only concerns primes $p \geq 7$. We need a choice of representation $\Gamma$ of $A_5$ in $\PGL_2(\Fpbar)$. We first consider the representation over $\bar{\Q}$ that was used by Klein \cite{klein_icosa}; denoting a chosen primitive fifth root of unity by $\zeta$, the matrix generators are
\begin{equation*}
S = \begin{bmatrix}
\zeta^3 & 0 \\
0 & \zeta^2
\end{bmatrix}, \quad
T =
\begin{bmatrix}
\zeta - \zeta^4 & -\zeta^2 + \zeta^3 \\
-\zeta^2 + \zeta^3 & \zeta - \zeta^4
\end{bmatrix}.
\end{equation*}
Next we verify that the reduction mod $p$ homomorphism is injective. Since $A_5$ is simple, the possibilities for the kernel are the trivial group and all of $A_5$, and the kernel does not contain $S$ as long as $p \neq 5$, so the kernel in our case is trivial.

The Appendix shows that, if a map has automorphism group strictly larger than $A_5$, then its automorphism group is $p$-irregular, so the same method as the previous section applies.

Doyle and McMullen provide examples of maps with exact automorphism group $A_5$ over $\bar{\Q}$ \cite{DM}. We try
\begin{equation*}
    f(z) = \frac{z^{11} + 66z^6 - 11z}{-11z^{10} - 66z^5 + 1}.
\end{equation*}
The resultant is divisible only by 2, 3, and 5, so the reduced map is degree 11 for all $p > 5$. Equation \eqref{degtest} shows that the only primes for which we may pick up extra automorphisms are $p = 2,3,5,11$. In fact, when $p = 11$, our example reduces to $z^{11}$, which has exact automorphism group $\PGL_2(\FF_{11})$.

So for the case $p = 11$, we try a different map. We check
\begin{equation*}
    f(z) = \frac{-57 z^{15} + 247 z^{10} + 171 z^5 + 1}{-z^{19} + 171z^{14} - 247z^{9} - 57z^4}.
\end{equation*}
We confirm that 11 does not divide the resultant, so the map is degree 15 after reduction; then \eqref{degtest} shows that $f_{11}$ has $p$-regular automorphism group, and we compute $\Aut(f_{11}) \cong A_5$.
\end{proof}

\begin{rmk}
    When $p = 3$, there is a subgroup of $\PGL_2(\bar{\FF}_3)$ isomorphic to $A_5$, but it is $p$-irregular. In Section \ref{sect_irreg_a5}, we checked that $A_5$ is realizable when $p = 3$ directly, rather than by reducing a map over $\bar{\QQ}$ modulo 3.
\end{rmk}

The tetrahedral group $A_4$ is a bit more difficult to analyze than the previous cases because the representations of $A_4$ in $\PGL_2(\bar{\FF}_p)$ are subrepresentations of $S_4$, which is also $p$-regular. There is also an additional curiosity in that the maps which invariant theory furnishes over $\bar{\Q}$ are not defined over $\Q$, for the particular representation we work with. This does not affect our calculation, but it is interesting.

\begin{proof}[Proof of Theorem \ref{thm_realizable_p_regular} part \eqref{p_regular_tetrahedral}]
Since $|A_4| = 12$, we work with $p \geq 5$. Let $\Gamma$ be the $\bar{\Q}$-representation of $A_4$ with matrix generators
\begin{equation*}
\left\{\begin{bmatrix}
i & i \\
1 & -1 \\
\end{bmatrix},\quad
\begin{bmatrix}
-1 & 0 \\
0 & 1 \\
\end{bmatrix}, \quad
\begin{bmatrix}
0 & 1 \\
1 & 0
\end{bmatrix}
\right\},
\end{equation*}
where $i$ is a primitive fourth root of unity. In the proof of Theorem \ref{thm_realizable_p_regular} part (\ref{p_regular_octahedral}), we showed that the reduction map is injective for this representation.

DeFaria-Hutz \cite{Hutz15} provides examples of maps over $\bar{\Q}$ with exact automorphism group $\Gamma$. We first try
$$ f(z) = \frac{\sqrt{-3}z^2 - 1}{z^3 + \sqrt{-3}z}.$$
The resultant has just 2 as a prime factor, so for $p > 2$ the degree is still 3 after reduction.

Next we check against extra automorphisms. The argument of Faber \cite[Proposition 4.14, 4.17]{Faber} shows that each tetrahedral subgroup $\Gamma$ of $\PGL_2(\bar{\FF}_p)$ is uniquely contained in an octahedral group. The cited argument starts with a particular choice of $\Gamma$ and calculates the copy of $S_4$; since the argument uses a different choice of $\Gamma$ than we do, we are using the fact that every tetrahedral subgroup is conjugate in $\PGL_2(\bar{\FF}_p)$.

In our case the octahedral group is, as described previously, generated by $\Gamma$ together with
\begin{equation*}
\begin{bmatrix}
i & 0 \\
0 & 1
\end{bmatrix}.
\end{equation*}
We check directly that this matrix is not an automorphism of $f$, even after reduction, by starting from the equation $f(iz) = if(z)$ and simplifying. The calculation is omitted.

Because the automorphism group of $f$ is not isomorphic to $S_4$, if there were remaining automorphisms, then the automorphism group would be $p$-irregular. The test of equation \eqref{degtest} shows that $f_p$ has $p$-regular automorphism group except possibly when $p = 2, 3$, and these primes are not present in this case.

To prove that $f$ is of minimal degree, we need only rule out the possibility that a degree 2 map has automorphism group $A_4$. This follows from \eqref{degtest}, since $A_4$ contains an element of order 4.
\end{proof}


\section{ Theoretical tools for discovering examples }
\label{sect_dm}

In Theorem \ref{thm_realizable_p_regular} and Theorem \ref{thm_realizable_p_irregular}, we showed through explicit constructions that, for any prime power $q$, every subgroup of $\PGL_2(\FF_q)$ arises as the automorphism group of a dynamical system. For instance, we calculated that $\PSL_2(\FF_q)$ is the automorphism group of a certain dynamical system of degree at most $\frac{1}{2}(q^3 - 2q^2 + q + 2)$. In this section, we develop the theoretical tools that explain how we arrived at these constructions. We present the motivating theorems, then develop the proofs in stages.

Our work is modeled on the theory over $\C$. In work on the quintic, Doyle and McMullen \cite{DM} proved a version of the following structure theorem for rational maps of $\PP^1(\C)$ with automorphisms. The theorem statement requires definitions from invariant theory, which we defer to Section \ref{sect_prelim_invar}.

\begin{thm}[{{Doyle-McMullen \cite[Theorem 5.2]{DM}}}] \label{thm_dm}
Suppose that $\Gamma$ is a subgroup of $\PGL_2(\CC)$. Let $\hat{\Gamma}$ be the preimage of $\Gamma$ in $\SL_2(\CC)$. Then every rational map $f$ such that $\deg(f) \geq 2$ and $\Gamma \subseteq \Aut(f)$ arises in the form
$$[x : y] \mapsto \left[xF + \dd{G}{y} : yF - \dd{G}{x}\right],$$
where $F$ and $G$ are homogeneous, relatively invariant polynomials for the same character of $\hat{\Gamma}$, such that $F=0$ or $\deg(F) + 1 = \deg(G) - 1 = \deg(f)$.
\end{thm}
The proof idea is that there are ways of going back and forth (not quite bijectively) between the following sets:
\begin{itemize}
    \item Rational maps of $\PP^1(\C)$ such that $\Gamma \subseteq \Aut(f)$;
    \item Homogeneous invariant polynomial differential 1-forms in $x, y$ over $\C$;
    \item Pairs $(F, G)$ of homogeneous invariant polynomials in $\C[x,y]$ such that $F=0$ or $\deg(F) + 2 = \deg(G)$.
\end{itemize}
In characteristic $p$, both the proofs and the results require modification, mainly because not all polynomials have antiderivatives. We prove the following variation.

\begin{thm} \label{thm_dmcharp}
Let $p$ be a prime, and let $q$ be a power of $p$.
\begin{enumerate}
    \item Suppose that $p > 2$ and $\Gamma$ is a $p$-irregular subgroup of $\PGL_2(\FF_q)$. Let $\hat{\Gamma}$ be the preimage of $\Gamma$ in $\SL_2(\FF_{q^2})$. Then every rational map $f$ such that $\deg(f) \geq 2$ and $\Gamma \subseteq \Aut(f)$ arises in the form
    \begin{equation}\label{eq_map_form}
        [x : y] \mapsto \left[ xF + \frac{\partial G}{\partial y} : yF - \frac{\partial G}{\partial x} \right],
    \end{equation}
    where $F$ and $G$ are homogeneous, relatively invariant polynomials over $\Fpbar$ for the same character of $\hat{\Gamma}$, such that $\deg(F) + 1 = \deg(G) - 1 = \deg(f)$.
    \item Let $p \geq 2$. Let $F$ and $G$ be homogeneous, relatively invariant polynomials over $\Fpbar$ for the same character of $\Gamma \subseteq \SL_2(\FF_q)$. Let $\bar{\Gamma}$ be the image of $\Gamma$ in $\PSL_2(\FF_q)$. If the expressions $xF + \dd{G}{y}$ and $yF - \dd{G}{x}$ are nontrivial homogeneous polynomials of the same degree, then the corresponding rational map $f$ of the form \eqref{eq_map_form} has $\bar{\Gamma} \subseteq \Aut(f)$.
\end{enumerate}
\end{thm}

One of the difficulties in applying Theorem \ref{thm_dmcharp} to the realizability problem is that the resulting map $f$ may only satisfy $\Gamma \subseteq \Aut(f)$, while we are looking for equality. In trying to realize $\PSL_2(\FF_q)$, for many choices of invariants, the machinery of Theorem \ref{thm_dmcharp} resulted in a map with automorphism group $\PGL_2(\FF_q)$. We observed that the degree of the minimal example of exact $\PSL_2(\FF_q)$ automorphism group was a cubic polynomial in $q$. We formulate this observation as Theorem \ref{thm_pslminimal}, with the tools for the proof coming from Propositions \ref{rationalmapstoforms} and \ref{dmcharp}.

\begin{thm}\label{thm_pslminimal}
Let $p > 2$ and let $q$ be a power of $p$. The degree of a rational map with automorphism group $\PSL_2(\FF_q)$ must be at least
\begin{equation*}
    \frac{1}{2} \left( q^3 - 2q^2 + q + 2 \right).
\end{equation*}
\end{thm}

We omit the case $p = 2$ because then $\PSL_2(\FF_q)$ and $\PGL_2(\FF_q)$ coincide.

A second application of Theorem \ref{thm_dmcharp} is to the problem of realizing $A_5$ over $\bar{\FF}_3$. Over $\CC$, the minimal degree of a map with exact automorphism group $A_5$ is 11. The example over $\bar{\FF}_3$ described in Theorem \ref{thm_realizable_p_irregular} part \ref{thm_realizable_irreg_a5} has degree 21. This turns out to be the minimal degree for $A_5$ when $p = 3$.
\begin{thm} \label{thm_a5_minimal}
    A rational map over $\bar{\FF}_3$ with automorphism group $A_5$ has degree at least $21$.
\end{thm}

\subsection{Preliminaries from invariant theory} \label{sect_prelim_invar}
Let $k$ be a field. Let $V$ be a 2-dimensional vector space over $k$. Let $H$ be a subgroup of $\GL_2(k)$. Let $P[V]$ be the algebra of polynomial functions on $V$, that is, the symmetric algebra of the dual space $V^*$. Then $H$ acts on $P[V]$ by \emph{pullback},
$$ H \times P[V] \to P[V],$$
$$ (h, F) \mapsto F \circ h.$$
We write $h^* F = F \circ h$.
The elements of $P[V]$ fixed by this action form a subring of $P[V]$, denoted $P[V]^H$, called the \emph{ring of (polynomial) invariants}.

Let $\chi$ be a character of $H$, that is, a homomorphism $H \mapsto k^{\ast}$. The set
$$\{ F \in P[V] : \forall h\in H, h^* F = \chi(h) F \} $$
forms a $P[V]$-submodule of $P[V]$ called the \emph{module of relative (polynomial) invariants}, denoted $P[V]^{H}_{\chi}$.

We make the analogous definitions for formal differential forms, following Smith \cite{smith}. Let $\Lambda[V]$ be the exterior algebra on the dual space $V^*$. Let $E[V] = P[V] \otimes \Lambda[V]$. The algebra $E[V]$ is called the \emph{polynomial tensor exterior algebra}. We think of its elements as formal differential forms defined only with polynomials.

We recall the basic properties of $E[V]$, for convenience choosing a basis $v, w$ of $V$.
\begin{enumerate}
    \item The basis $v, w$ of $V$
        induces a basis $x, y$ of $P[V]$ and algebra generators $dx, dy$ of $\Lambda[V]$. The exterior algebra $\Lambda[V]$ is spanned as a $k$-vector space by $1, dx, dy, dx \wedge dy$. The polynomial algebra $P[V]$ is infinite-dimensional as a $k$-vector space and is spanned by monomials in $x, y$.
    \item There is a $k$-linear map $d \colon E[V] \to E[V]$ called the \emph{exterior derivative}. It is defined as follows. We set
        $$d(dx) = d(dy) = 0,$$
        $$d(x) = dx,$$
        $$d(y) = dy.$$
        Then we extend $d$ to all of $E[V]$ by linearity and the Leibniz rule
        $$d(\theta_1 \theta_2) = (d \theta_1) \theta_2 + \theta_1 (d \theta_2).$$
        In particular, for any $f \in P[V]$, we have
        $$df = \dd{f}{x} dx + \dd{f}{y} dy. $$
        Forms in the kernel of $d$ are \emph{closed}, and forms in the image of $d$ are \emph{exact}. For any $\omega \in E[V]$, we have $d(d \omega) = 0$, so exact forms are closed.
    \item
        A group $H$ of linear self-maps of $V$ induces a pullback action on $E[V]$, as we now explain. Elements of $E[V]$ are sections of the bundle of differential forms on $V$, where $V$ is viewed as a variety. Let $TV$ be the tangent bundle on $V$. Let $D$ denote the standard Jacobian matrix derivative. Then any algebraic map $h: V \to V$ induces a pushforward map $h_* : TV \to TV$. It is defined as follows: given a tangent vector $\delta$ to a point $v \in V$, the pushforward $h_*\delta$ is the tangent vector $(Dh)(\delta)$ to $h(v)$. Since $V$ is a vector space, we may canonically identify the tangent spaces $(TV)_v$ and $(TV)_{h(v)}$ with $V$. In our setting, the self-map $h$ is linear rather than just algebraic, so it is equal to its own Jacobian with this identification. Thus
        $$h_* \delta = h(\delta).$$
        The pullback of a form $\theta \in E[V]$ by an algebraic map $h \colon V \to V$ is defined by
        $$h^* \theta = \theta \circ h_*.$$
        Thus any group $H$ of linear self-maps of $V$ induces an action on $E[V]$. It follows from the definition that $h^*$ respects the algebra structure of $E[V]$ and that $h^*$ commutes with the exterior derivative $d$.
\end{enumerate}

A form $\omega \in E[V]$ is called \emph{relatively invariant} for $H$ (with respect to a character $\chi$) if for all $h \in H$, we have
\begin{equation*}
    h^{\ast} \omega = \chi(h) \omega.
\end{equation*}
If $\chi$ is the trivial character, then $\omega$ is also called an \emph{absolute} invariant.

The set of relatively invariant forms for $H$ with character $\chi$ form the \emph{module of relatively invariant (formal differential) forms}, denoted $E[V]^{H}_{\chi}$.

There are a number of natural gradings to consider on $E[V]$. Our convention for the grading is as follows. After choosing generators $x, y$ for $P[V]$ and the corresponding basis $dx, dy$ for the 1-forms in the exterior algebra, we assert that $x$ and $y$ have degree 1 and that $dx$ and $dy$ have degree 0; then we extend multiplicatively. In particular, a homogeneous 1-form is one where $dx$ and $dy$ have coefficients which are homogeneous polynomials of the same degree. (Our convention is that $0$ is of every degree.)

\subsection{From rational maps to 1-forms and back}

We follow \cite[Section 5.III]{DM}. Viewing $V$ as a variety, each tangent space of $V$ is canonically isomorphic to $V$. Thus, given any polynomial map
\begin{align*}
    &\Phi: \A^2 \to \A^2,\\
    &\Phi(x,y) = (\Phi_1(x,y), \Phi_2(x,y)),
\end{align*}
we can associate a vector field $X_\Phi$ on $V$; it sends $(x,y)$ to the point in $(TV)_{(x,y)}$ corresponding to $\Phi(x,y)$. For any linear map $h \colon V \to V$, we may consider the pushforward $h_* X_\Phi$ and the conjugate map $\Phi^h$. It follows immediately from the definition of $X_\Phi$ that in fact,
$$h_* X_\Phi = X_{\Phi^h}.$$
Throughout, let $\omega = dx \wedge dy$. Let $\omega_\Phi$ be the 1-form defined by contraction of $\omega$ by the vector field $X_\Phi$; that is,
$$\omega_\Phi( \cdot ) = \omega( X_\Phi, \cdot ).$$
In coordinates,
$$\omega_\Phi = \Phi_2 dx - \Phi_1 dy.$$
It follows from the definition of $\omega_\Phi$, or its expression in coordinates, that for any invertible linear map $h \colon V \to V$, we have
$$h^\ast \omega_\Phi = \omega_{\Phi^h}.$$
In particular, we have $\Phi = \Phi^h$ if and only if $h^\ast \omega_\Phi = \omega_\Phi$.

We can use this connection to translate data about the automorphism group of $f$ into invariant theory, as follows.

Let $\Gamma \subseteq \PGL_2(k)$. Let $\hat{\Gamma}$ be a subgroup of $\GL_2(k)$ that is mapped to $\PGL_2(k)$ by projectivization. Let $f$ be a rational map. If $\gamma \in \Gamma$ is an automorphism of $f$, then $f^\gamma = f$. Let $\Phi$ be any lift of $f$ to a polynomial function on $\A^2$. Specifically, $\Phi$ is a pair of homogeneous polynomials that define the same endomorphism of $\PP^1$ as $f$. Let $M$ be any preimage of $\gamma$ in $\hat{\Gamma}$. Since $f = f^\gamma$, there exists some value $\chi(M) \in k^{\ast}$ such that $\Phi^M = \chi(M) \Phi$. In fact, $\chi(M)$ is independent of the choice of lift $\Phi$. The rule $M \mapsto \chi(M)$ defines a character $\chi: \hat{\Gamma} \to k^{\ast}$.
We have
\begin{equation*}
    M^{\ast} \omega_{\Phi} = \omega_{\Phi^M} = \omega_{\chi(M) \Phi} = \chi(M) \omega_{\Phi}.
\end{equation*}
So, if $f$ has automorphism group containing $\Gamma$, then for any lift $\Phi$ of $f$, the 1-form $\omega_{\Phi}$ is a relative invariant of $\hat{\Gamma}$ with respect to some character.

Conversely, to a nonzero homogeneous 1-form $\omega = f_1 dx + f_2 dy$, we can associate the rational map $r(\omega) := [-f_2 : f_1]$. If $\omega$ is relatively invariant for a subgroup $H$ of $\GL_2$, then the elements of the image $\bar{H}$ of $H$ in $\PGL_2$ are automorphisms of $r(\omega)$.
We have established the following proposition.

\begin{prop} \label{rationalmapstoforms}
Let $\Gamma \subseteq \PGL_2(k)$ and let $\hat{\Gamma}$ be a subgroup of $\GL_2(k)$ that maps to $\Gamma$ by projectivization. Let $f$ be a rational map of $\PP^1$.
\begin{enumerate}
    \item If $f$ has automorphism group containing $\Gamma$, then for any lift $\Phi$ of $f$, the 1-form $\omega_{\Phi}$ is a relative invariant of $\hat{\Gamma}$ with respect to some character.
    \item If $\omega_{\Phi}$ is a relative invariant of a group $H$, then $\bar{H} \subseteq \Aut(f)$.
\end{enumerate}
\end{prop}

Some remarks:

\begin{enumerate}
\item These associations, from rational maps to nonzero homogeneous 1-forms and back, are almost inverse, but not quite. There is no well-defined association $f \mapsto \omega_\Phi$, except up to scaling. Even so, we can say $r(\omega_{\Phi}) = f$.
\item We have $\deg(\omega_{\Phi}) = \deg(f)$. But because of the possibility of a common factor, the most we can say about $r(\omega)$ is that $\deg(r(\omega)) \leq \deg(\omega)$. Equality occurs if and only if $\omega$ has no nonzero homogeneous polynomial of positive degree as a factor.
\end{enumerate}

\subsection{From 1-forms to polynomials and back}

The next proposition links invariant 1-forms to pairs of invariant polynomials. We defer the proof to the end of this subsection.

\begin{prop} \label{dmcharp}
Let $\lambda = y dx - x dy$.
\begin{enumerate}
    \item \label{dmcharp_formtopolys}
Let $k$ be a field of characteristic $p$. Let $\eta$ be a homogeneous 1-form of degree $n$, where
\begin{equation} \label{eq_form_degree_hypoth}
    n \not\equiv -1 \pmod{p}
\end{equation}
Then there exist homogeneous polynomials $F$ and $G$, possibly 0, such that
\begin{equation*}
\eta = F \lambda + dG,
\end{equation*}
where $dG$ is the $1$-form $dG = \frac{\partial G}{\partial x}dx + \frac{\partial G}{\partial y} dy$.
Writing $\eta = \eta_1 dx + \eta_2 dy$, explicit formulas for $F$ and $G$ are
\begin{align*}
    F &= \frac{1}{n + 1}\left( \dd{\eta_1}{y} + \dd{\eta_2}{x} \right),\\
    G &= \frac{1}{n + 1}(x \eta_1 + y \eta_2).
\end{align*}

\item \label{dmcharp_formtopolysinv} Suppose $H$ is a subgroup of $\SL_2(k)$. If $\eta$ is a relative invariant for $H$ with character $\chi$, then the above $F$ and $G$ may further be chosen to be relative invariants of $H$ for character $\chi$.

\item \label{dmcharp_polystoform} Suppose $H$ is a subgroup of $\SL_2(k)$. If $F$ and $G$ are homogeneous invariant polynomials of $H$ with character $\chi$ such that $F \lambda + d G$ is homogeneous, then $F \lambda + dG$ is also a relative invariant for $\chi$.
\end{enumerate}
\end{prop}

\begin{rmk}
To show that the degree hypothesis \eqref{eq_form_degree_hypoth} is needed, consider the example $\eta = y^{p - 1} dx$. If we assume $\eta = F \lambda + d G$ for some $F, G$, then we get the equations
\begin{align*}
    y^{p - 1} &= yF + \dd{G}{x},\\
    0 &= -xF + \dd{G}{y}.
\end{align*}
An appropriate linear combination of the above equations gives
\begin{equation*}
    x y^{p - 1} = x \dd{G}{x} + y \dd{G}{y} = (\deg G) G = 0,
\end{equation*}
which is false.

The restriction on degree makes this proposition more subtle than its characteristic 0 counterpart. But in our application (Theorem \ref{thm_dmcharp}), the degree hypothesis is automatically satisfied in the $p$-irregular case. Thus, Proposition \ref{dmcharp} is a rare example of modular invariant theory being \emph{less} complicated than nonmodular invariant theory.
\end{rmk}

\begin{rmk}
There are creative ways of evading the degree hypothesis \eqref{eq_form_degree_hypoth}. For instance, say $p > 2$ and $\eta$ is a relative invariant for $H$ with character $\chi$ with degree $n$, where
$$n \equiv -1 \pmod{p}.$$
There is an absolutely invariant homogeneous polynomial of $\GL_2(\FF_q)$ of degree $q^2 - 1$, which we denote $u$ (see, for instance, Smith \cite[Chapter 8]{smith}). Then $u \eta$ is a relative invariant for $H$ with character $\chi$ with degree $-2 \mod p$. Thus $\eta$ can be written in the form $(F\lambda + dG)/u$, where $F$ and $G$ are in degrees $n + q^2 - 2$ and $n + q^2$, respectively. Thus, the structure of the module of relative invariants still affects the existence of rational maps in these degrees.
\end{rmk}

Before we embark on the proof, we first need a version of the Poincar\'e Lemma of exterior algebra that is appropriate for fields of characteristic $p$.

\begin{lem} \label{poincare} Say $\eta$ is a homogeneous, closed 1-form on a 2-dimensional vector space over a field of characteristic $p$. Suppose also that $\eta$ has degree $n$ such that
$$n \not\equiv -1 \pmod{p}.$$
Then $\eta$ is exact.
\end{lem}

\begin{proof}
Express $\eta$ in a basis as $\eta_1 dx + \eta_2 dy$. Since $\eta$ is closed, we may obtain from the equation $d \eta = 0$ that
$$ \dd{\eta_1}{y} = \dd{\eta_2}{x}.$$
Then we compute explicitly
\begin{align*}
    d(x \eta_1 + y \eta_2) &= \eta_1 dx + x d \eta_1 + \eta_2 dy + y d \eta_2 \\
    &= \eta + xd \eta_1 + yd  \eta_2 \\
    &= \eta + x \frac{\partial \eta_1}{\partial x} dx + x \frac{\partial \eta_1}{\partial y} dy + y \frac{\partial \eta_2}{\partial x} dx + y \frac{\partial \eta_2}{\partial y} dy \\
    &= \eta + x \frac{\partial \eta_1}{\partial x} dx + y \frac{\partial \eta_1}{\partial y} dx + x \frac{\partial \eta_2}{\partial x} dy + y \frac{\partial \eta_2}{\partial y} dy & \text{(using closedness)} \\
    &= \eta + n \eta = (n + 1) \eta.& \text{(using homogeneity)}.
\end{align*}
By assumption we may divide by $n + 1$, so we have the explicit formula
\begin{equation} \label{eq_poincare}
    \eta = d \left( \frac{1}{n+1}(x \eta_1 + y \eta_2) \right).
\end{equation}
\end{proof}

We are ready to prove Proposition \ref{dmcharp}.

\begin{proof}[Proof of Proposition \ref{dmcharp}]
Throughout, set $\omega = dx \wedge dy$. Notice that $\omega$ is absolutely invariant with respect to $\SL_2(k)$.
\begin{enumerate}
\item One may just check that the given formulas for $F$ and $G$ suffice. We now explain how to derive the formulas. First, we show that there is a homogeneous polynomial $F$ of degree $n - 1$ such that $d \eta = d (F \lambda)$. We have
$$d \eta = \left( -\dd{\eta_1}{y} + \dd{\eta_2}{x} \right) \omega.$$
For convenience, let
$$j = \left( -\dd{\eta_1}{y} + \dd{\eta_2}{x} \right).$$
Thus
$$d \eta = j \omega.$$
For any homogeneous polynomial $F$, we have by the Leibniz rule that
\begin{align*}
    d (F \lambda) &= (dF)\lambda + F(d\lambda) \\
    &= \left(\dd{F}{x} dx + \dd{F}{y} dy \right)(-ydx + xdy) + F \omega \\
    &= \left( \dd{F}{x} x + \dd{F}{y} y \right) \omega + 2 F \omega \\
    &= (2 + \deg F)F \omega.
\end{align*}
Thus, the desired $F$ must satisfy
$$F = \frac{j}{2 + \deg F},$$
so we see that $F$ must be of degree $n - 1$ and we take
$$F = \frac{j}{n+1}.$$
With this choice of $F$, we know that the 1-form $\eta - F\lambda$ is closed, hence exact by Lemma \ref{poincare}. Thus, there exists a 0-form $G$ such that $dG = \eta - F\lambda$. Applying \eqref{eq_poincare}, we obtain the stated formula for $G$.

\item Let $h \in H$. Let $j$ be as in the proof of (1). We compute the pullback $h^\ast (d \eta)$ two ways. On one hand,
\begin{align*}
    h^\ast (d \eta) &= h^\ast( j \omega) & \\
    &=
    (h^\ast j) (h^\ast \omega) & (h^* \text{ respects multiplication}) \\
    &= (h^\ast j) \omega. & (\omega \text{ is absolutely invariant).}\\
\end{align*}
On the other hand,
\begin{align*}
    h^\ast (d \eta) &= d(h^\ast \eta) & (d \text{ and } h^\ast \text{ commute} ) \\
    &= d(\chi(h) \eta) & (\eta \text{ is relatively invariant}) \\
    &= \chi(h) d\eta & \\
    &= \chi(h) j \omega.\\
\end{align*}
Thus $h^\ast j = \chi(h) j$, so $j$ is relatively invariant. Since $F = j/(n+1)$, we conclude that $F$ is relatively invariant.

Now we show that $G$ is relatively invariant, that is, that $\chi(\gamma) G = \gamma^{\ast} G$.

Since $\eta$ and $F$ are relatively invariant, and $\lambda$ is absolutely invariant, we have
\begin{align*}
    h^\ast (dG) &= h^\ast(\eta - F \lambda) \\
    &= h^\ast \eta - (h^\ast F)(h^\ast \lambda) \\
    &= \chi(h) \eta - \chi(h) F \lambda \\
    &= \chi(h)(dG).
\end{align*}
So $dG$ is relatively invariant. This implies that $\chi(h) G - h^\ast G$ is a homogeneous closed 0-form of degree $n + 1$. The only nonzero closed 0-forms are elements of the polynomial ring $k[x^p,y^p]$.
By the assumption that $n \not\equiv - 1$ mod $p$, we may conclude that $\chi(h) G - h^\ast G = 0$, so $G$ is relatively invariant.

\item Let $h \in H$. We compute
\begin{align*}
    h^{\ast}(F \lambda + dG) &= (h^{\ast} F)( h^{\ast} \lambda) + h^{\ast} (dG) \\
    &= \chi(h) F h^{\ast} \lambda + h^{\ast} (dG) \\
    &=  \chi(h) F \lambda + h^{\ast} (dG) \\
    &= \chi(h) F \lambda + d(h^{\ast}G) \\
    &= \chi(h) F \lambda + d(\chi(h)G) \\
&= \chi(h)( F \lambda + dG).
\end{align*}
\end{enumerate}
\end{proof}

\subsection{Proofs} \label{sect_dm_proofs}
    We conclude this section by proving Theorem \ref{thm_dmcharp}, Theorem \ref{thm_pslminimal}, and Theorem \ref{thm_a5_minimal}.

\begin{proof}[Proof of Theorem \ref{thm_dmcharp}]
\hfill
\begin{enumerate}
    \item Let $f$ be a map as described in the theorem statement. Choose any lift $\Phi$ of $f$. Then $\deg(\omega_{\Phi}) = \deg(f)$. By equation \eqref{degtest}, we know $\deg(\omega_{\Phi}) \equiv -1, 0, 1 \mod p$. Since $p > 2$ by assumption, the form $\omega_{\Phi}$ meets the degree hypothesis of Proposition \ref{dmcharp} (\ref{dmcharp_formtopolys}). Since $\omega_{\Phi}$ is the form associated to a rational map via Proposition \ref{rationalmapstoforms}, it is relatively invariant for $\hat{\Gamma}$ with respect to some character, so the invariance hypothesis of Proposition \ref{dmcharp} (\ref{dmcharp_formtopolysinv}) is also met. To meet the hypothesis that $\hat{\Gamma}$ is a subgroup of $\SL_2(k)$, we take $k$ to be $\FF_{q^2}$ and view $\Gamma$ as a subgroup of $\PSL_2(\FF_{q^2})$. Thus, we can write $\omega_{\Phi} = F \lambda + d G$ for relative invariant homogeneous polynomials $F, G$ for the same character. Then, again by Proposition \ref{rationalmapstoforms}, we have
    \begin{equation*}
        f = r(\omega_{\Phi}) = r(F \lambda + dG).
    \end{equation*}
    The theorem statement is just this equation written in coordinates.

    \item Let $\omega = F \lambda + dG$. The conditions on $F$ and $G$ ensure that $\omega$ is homogeneous and nonzero. By Proposition \ref{dmcharp} \eqref{dmcharp_polystoform}, $\omega$ is relatively invariant for $\Gamma$. Then $r(\omega)$ has the claimed automorphisms, by the discussion immediately preceding Proposition \ref{rationalmapstoforms}.
\end{enumerate}
\end{proof}

\begin{proof}[Proof of Theorem \ref{thm_pslminimal}]
    Assume that $f$ has automorphism group $\PSL_2(\FF_q)$. Write $d = \deg(f)$. Let $\omega$ be a 1-form associated to $f$ via Proposition \ref{rationalmapstoforms}. By the proof of Proposition \ref{dmcharp}, there exist relatively invariant homogeneous polynomials $F, G$ of $\SL_2(\FF_q)$ such that $\omega = F \lambda + d G$. We also know $\deg F + 1 = \deg G - 1 = d$ (or $F=0$). Surely $G \neq 0$ because otherwise there would be a homogeneous factor, causing $f$ to be degree 1, which we reject.

For $q > 2$, the only character of $\SL_2(\FF_q)$ is the trivial character. To see this, we invoke a well-known fact from group theory (see Dickson \cite{Dickson}): the abelianization of $\SL_2(\FF_q)$ is trivial as long as $q \geq 4$. Every character factors through the abelianization, so every character is trivial. For $q=3$, the group $\PSL_2(\FF_3)$ is isomorphic to the alternating group $A_4$. There are only two $3$-regular conjugacy classes in $A_4$, so there are two modular characters. These are the trivial character and a degree $3$ character (the reduction of the ordinary degree $3$ character). Since we are only interested in linear characters for invariants, we need only consider the trivial character in the $q=3$ case.

Now we ask for which values of $d$ there exist homogeneous invariant polynomials in degrees $d - 1$ and $d + 1$. We cite a standard theorem in modular invariant theory, see Smith \cite[Theorem 8.1.8]{smith}: the ring of invariants $\FF_q[x,y]^{\SL_2(\FF_q)}$ is generated as an $\FF_q$-algebra by the fundamental invariants
$$u_1 = x^q y - x y^q$$
and
$$u_2 = \sum_{n = 0}^q x^{(q-1)(q-n)}y^{(q-1)n} = \frac{x^{q^2}y - x y^{q^2}}{x^q y - x y^q}.$$

The set of degrees of nontrivial polynomial invariants is, thus, the numerical semigroup generated by $q + 1$ and $q(q-1)$. It is also known that $u_1$ and $u_2$ are algebraically independent; that is, the ring of invariants above is actually a polynomial ring. So we can write $F$ and $G$ as polynomials in $u_1$ and $u_2$, in a unique way.

Next, we show that certain simple families of $F$ and $G$ give rise to 1-forms which are relatively invariant for a character of $\GL_2(\FF_q)$. By Proposition \ref{rationalmapstoforms}, such 1-forms give rise to rational maps with automorphism group $\PGL_2(\FF_q)$. Therefore, the only way to get a map with exact automorphism group $\PSL_2(\FF_q)$ is to avoid these families.

The determinant, $\det$, is a character of $\GL_2(\FF_q)$. The polynomial $u_1$ and the $1$-form $\lambda$ are, by direct calculation, relative invariants for $\det$. The polynomial $u_2$ is an absolute invariant of $\GL_2(\FF_q)$. This causes many simple expressions of the form $F \lambda + d G$ to be relative invariants for some power of $\det$.

Each pair of $F$ and $G$ falls into at least one of the following cases:

\begin{enumerate}
\item $F = 0$.
\item $F \neq 0$ and $F$ and $G$ are monomials in $u_1, u_2$.
\item At least one of $F$ and $G$ is not a monomial in $u_1, u_2$.
\end{enumerate}

Now we see which elements of these cases are admissible, in the sense that $F \lambda + d G$ is not a relative invariant for any power of $\det$.

\begin{enumerate}
\item Say $F = 0$. If $G$ is a pure polynomial in $u_1$, it is of the form ${c {u_1}^k}$, so it is a relative invariant of $\GL_2(\FF_q)$ for $\det^k$. If $G$ is a pure polynomial in $u_2$, it is an absolute invariant of $\GL_2(\FF_q)$. So $G$ must contain a binomial, which reduces us to the last case.
\item Write $F = \alpha {u_1}^{a_1} {u_2}^{a_2}, G = \beta {u_1}^{b_1} {u_2}^{b_2}$, where $\alpha, \beta \in \FF_q^{\ast}$. Then $F \lambda$ is relatively invariant for $\det^{a_1 + 1}$ and $G$ is relatively invariant for $\det^{b_1}$. The sum of relative invariants for the same character is again a relative invariant, so $\det^{a_1 + 1} \neq \det^{b_2}$. Since $\det^{q-1}$ is trivial by cyclicity of $k^{\ast}$, we conclude
\begin{equation*}
    a_1 + 1 \not\equiv b_1 \mod (q-1).
\end{equation*}
This property is preserved by multiplying or factoring out a monomial simultaneously from $F$ and $G$. Thus, we reduce to one of the following cases: $F = {u_1}^{a_1}$ and $G = {u_2}^{b_2}$, or $F = {u_2}^{a_2}$ and $G = {u_1}^{b_1}$.

In the first case, $a_1$ and $b_2$ are positive solutions to
\begin{equation*}
    a_1 (q + 1) + 2 = b_2 (q^2 - q).
\end{equation*}

Finding minimal solutions for such equations is a basic Diophantine problem. Reducing modulo $q(q-1)/2$, we find $a_1 \equiv q-2$. We earlier found that $a_2 + 1 \not\equiv b_1 \mod (q - 1)$, and $b_1 = 0$, so we cannot have $a_1 = q - 2$. Looking at the next positive solution for $a_1$ gives
\begin{equation*}
    a_1 \geq \frac{1}{2} q(q-1) + q - 2.
\end{equation*}
Then the degree of $f$ in this case is at least $\left( \frac{1}{2}q(q-1) + q - 2\right)(q + 1) + 1$.

In the second case, $a_2$ and $b_1$ are positive solutions to
\begin{equation*}
    a_2 (q^2 - q) + 2 = b_1 (q + 1).
\end{equation*}

The minimal solution occurs when
\begin{equation*}
   b_1 \geq \frac{1}{2}(q^2 - q) - q + 2.
\end{equation*}

Then
\begin{equation*}
    \deg(f) \geq \frac{1}{2}(q^3 - 2q^2 + q + 2).
\end{equation*}

\item The lowest degree homogeneous polynomial in $u_1$, $u_2$ that is not a monomial occurs in degree
\begin{equation*}
    \lcm(\deg(u_1), \deg(u_2)) = \frac{1}{2}q(q-1)(q+1).
\end{equation*}
Thus, if $F$ or $G$ contains a binomial, $\deg(f) \geq \frac{1}{2}q(q-1)(q+1) - 1$.
\end{enumerate}
Recalling that $q \geq 3$, the bound
\begin{equation*}
    \deg(f) \geq \frac{1}{2}(q^3 - 2q^2 + q + 2)
\end{equation*}
holds across all the cases.
\end{proof}

\begin{proof}[Proof of Theorem \ref{thm_a5_minimal}]
Assume that $f$ has automorphism group $A_5 \subseteq \PGL_2(\bar{\FF}_3)$. We claim that $\deg f \geq 49$. Let $\zeta$ be a primitive fifth root of unity in $\bar{\FF}_3$. By a conjugacy, we may assume that the subgroup $A_5$ is generated by the matrices
\begin{equation*}
\begin{bmatrix}
\zeta & 0 \\
0 & 1
\end{bmatrix} \quad \text{and} \quad
\begin{bmatrix}
1 & 1 - \zeta - \zeta^{-1} \\
1 & -1
\end{bmatrix}.
\end{equation*}
Let $\hat{A}_5$ be the inverse image of $A_5$ in $\SL_2(\bar{\FF}_3)$. We note that the only character of $\hat{A}_5$ is the trivial character. To see this, notice that $\hat{A}_5$ is isomorphic to $\SL_2(\bar{\FF_5})$, so the abelianization of $\hat{A}_5$ is trivial. By the lack of nontrivial characters and Theorem \ref{thm_dmcharp}, there exist homogeneous polynomials $F, G$ over $\bar{\FF}_3$ that are absolutely invariant for $\hat{A}_5$, such that either $F = 0$ or $\deg (F) + 1 = \deg(G) - 1 = \deg f$, and
$$f = \left[ xF + \dd{G}{y} : yF - \dd{G}{x} \right]. $$
In Magma \cite{magma}, we compute the fundamental invariants of $\hat{A}_5$. Let $i \in \bar{\FF}_3$ satisfy $i^2 + 1 = 0$; then the fundamental invariants are
    \begin{align*}
                u_1 &= x^{10} + i y^{10},\\
                u_2 &= x^{11}y + (i+2) x^6 y^6 - i xy^{11}.
    \end{align*}
We now rule out some low-degree possibilities for $F$ and $G$. We let $c_1, c_2$ denote arbitrary nonzero constants in $\bar{\FF}_3$.
\begin{itemize}
    \item If $F = 0$ and $G = c_1 u_1$, then by direct computation, the map $f$ has extra automorphisms.
    \item If $F = 0$ and $G = c_2 u_2$, then there is a common factor in the formula for $f$, so $\deg f < \deg(G) - 1$.
    \item If $F = c_1 u_1$ and $G = c_2 u_2$, then there is a common factor in the formula for $f$.
    \item If $F = 0$ and $G = c_1 u_1^2$, then there is a common factor in the formula for $f$.
\end{itemize}
The above cases rule out all the possibilities for $G$ such that $\deg G \leq 21$, proving the theorem.
\end{proof}


\section{Moduli space $\mathcal{M}_2$ and its symmetry locus} \label{sect_Ad}
    We are interested in determining the automorphism locus $\cA_2(\Fpbar) \subset \mathcal{M}_2(\Fpbar)$. It is known that $\mathcal{M}_2\cong \mathbb{A}^2$ via the explicit isomorphism $f \mapsto (\sigma_1,\sigma_2)$, where $\sigma_1$ and $\sigma_2$ are the first two elementary symmetric polynomials evaluated on the multipliers of the fixed points \cite{Silverman}. Any automorphism must permute the fixed points of a map, and can only permute fixed points with the same multipliers because multipliers are invariant under conjugation. Utilizing this fact, in characteristic 0 the locus $\cA_2 \subset \M_2(\C)$ is worked out in detail in \cite{Fujimura} but is also discussed in \cite{Milnor}. The starting point is that the discriminant of the multiplier polynomial
    \begin{equation}\label{eq3}
        x^3 - \sigma_1x^2 + \sigma_2x - (\sigma_1-2),
    \end{equation}
    where $\sigma_1$, $\sigma_2$, and $\sigma_3$ are the elementary symmetric polynomials evaluated at the multipliers of the three fixed points. For there to be a non-trivial automorphism, there must be two distinct fixed points with the same multiplier, so the multiplier polynomial
    vanishes if there is a nontrivial automorphism. The two components of the curve defined by the vanishing of \eqref{eq3} are then analyzed, only one of which corresponds to the existence of a nontrivial automorphism. This provides a description of $\cA_2\subset \mathcal{M}_2(\C)$ as a cuspidal cubic where every map has automorphism group $C_2$, except at the cusp, where it is $S_3$. In particular, in characteristic $0$, the locus $\cA_2$ is Zariski closed and irreducible. We proceed similarly in characteristic $p>0$ to arrive at Theorem \ref{thm_Ad}, which shows starkly different geometry in the $p=2$ case.

    As an element of $\PGL_2(\Fpbar)$, an automorphism is completely determined by specifying the images of three points. It follows that if a map has three distinct fixed point multipliers, the three fixed points are fixed by any automorphism, and the map has no nontrivial automorphisms. We first show that every map with two distinct fixed points with the same multiplier has a nontrivial automorphism. We use several times the basic fact that a fixed point has multiplicity one if and only if its multiplier is not one. Since the fixed points are the zeros of
    \begin{equation*}
        f(x) -x,
    \end{equation*}
    they are simple roots (multiplicity one) if and only if the derivative
    \begin{equation*}
        f'(x)-1
    \end{equation*}
    does not also vanish.
    \begin{lem}\label{lem_john}
        Let $f\in \Rat_2(\Fpbar)$. If $f$ has two distinct fixed points with the same multiplier, then there exists an automorphism which maps the two fixed points to each other and fixes the third.
    \end{lem}
    \begin{proof}
        Let $f \in \Rat_2(\Fpbar)$ be a rational map which has two fixed points with the same multiplier $\lambda$. Note that $\lambda \neq 1$, since otherwise each fixed point has multiplicity at least $2$, and there can only be three fixed points for a degree $2$ map when counted with multiplicity. Label the multipliers of the three (with multiplicity) fixed points as $\lambda_1$, $\lambda_2$, and $\lambda_3$. Recall $\lambda_1\lambda_2 = 1 \iff \lambda_1=\lambda_2=1$ even in positive characteristic since the relation $\sigma_1 = \sigma_3+2$ implies the (formal) identities
        \begin{equation*}
            (\lambda_1-1)^2 = (\lambda_1\lambda_2-1)(\lambda_1\lambda_3-1) \quad \text{and} \quad
            (\lambda_2-1)^2 = (\lambda_2\lambda_1-1)(\lambda_2\lambda_3-1).
        \end{equation*}
        So we are in the case that $\lambda_1\lambda_2\neq 1$ and, by the Normal Forms Lemma \cite[Lemma~5.3]{Silvermana}, the map $f$ must be conjugate to a map of the form
        \begin{equation*}
            \phi(z) = \frac{z^2 + \lambda z}{\lambda z + 1}.
        \end{equation*}
        Then, conjugation by $z \mapsto \frac{1}{z}$ is an automorphism that permutes the fixed points $0$ and $\infty$.
    \end{proof}

    \subsection{Automorphism locus over $\FF_p$, for $p\neq 2,3$}
        In this case, we can follow Fujimura-Nishizawa \cite[Proposition~1]{Fujimura}, since no coefficients that arise have prime divisors other than $2$ and $3$. For a map corresponding to the point $(\sigma_1,\sigma_2)$ to have a nontrivial automorphism, at least two multipliers must be equal. The multipliers are the roots of the polynomial
        \begin{equation}\label{mult}
            x^3 - \sigma_1x^2 + \sigma_2x - \sigma_1 + 2,
        \end{equation}
        which has multiple roots if and only if its discriminant is $0$. Therefore, there are at least two equal multipliers exactly at the vanishing of its discriminant, which is
        \begin{equation}\label{disc}
            (\sigma_2 - 2\sigma_1 + 3)(2\sigma_1^3 + \sigma_1^2\sigma_2 - \sigma_1^2 - 4\sigma_2^2 - 8\sigma_1\sigma_2 + 12\sigma_1 + 12\sigma_2-36).
        \end{equation}

        Note that this equivalence holds over any field. The polynomial \eqref{disc} is presented with two factors. The zero locus of the first, $\sigma_2 - 2\sigma_1+3$, is exactly the set of points corresponding to maps with a fixed point of multiplier $1$. This is because a fixed point multiplier $\lambda$ is a root of $\eqref{mult}$; substituting $1$ for $x$ yields $\sigma_2 - 2\sigma_1+3$. Following Milnor \cite{Milnor}, we call the vanishing locus of this polynomial $\Per_1(1)$, since the locus is the set of all conjugacy classes that have a fixed point with multiplier of $1$.

        We claim that the second curve, a cuspidal cubic denoted $S$, is the automorphism locus of quadratic rational maps over $\FF_p$ for $p > 3$.

        We use the fact that a multiplier of a fixed point is equal to $1$ if and only if it the fixed point occurs with multiplicity greater than $1$. The two curves have a unique point of intersection at $(\sigma_1,\sigma_2) = (3,3)$, which corresponds to a triple fixed point where $\lambda_1=\lambda_2=\lambda_3=1$. All other points on $\Per_1(1)$ correspond to maps with a double fixed point and a single fixed point. Again by the Normal Forms Lemma \cite[Lemma~5.3]{Silvermana}, maps with $\lambda_1=\lambda_2=1$ are conjugate to a map of the form
        \begin{equation*}
            f(z) = z + \frac{1}{z} + \sqrt{1-\lambda_3},
        \end{equation*}
        which has a double fixed point at infinity and a single fixed point at $\frac{-1}{\sqrt{1-\lambda_3}}$. Infinity has preimages 0 and itself; we know that automorphisms permute the set of fixed points and permute their preimages. The only possible element of $\PGL_2$ which could be an automorphism, then, is the map $z \mapsto \frac{1}{z}$, which is not an automorphism of $f$.

        It follows that any map with a nontrivial automorphism must lie on $S$. Points with exactly two equal multipliers have $C_2$ as their automorphism group by Lemma \ref{lem_john}. Points with all three multipliers equal must have $\sigma_1=3\lambda$ and $\sigma_3 = \lambda^3$, so the multiplier must be a root of the polynomial
        \begin{equation}\label{3equal}
            x^3 - 3x + 2.
        \end{equation}
        This factors as $(x+2)(x-1)^2$, so there are only two points on $S$ with triple multipliers: $(\sigma_1,\sigma_2) \in \{(-6,12), (3,3)\}$. The point $(\sigma_1,\sigma_2) = (-6,12)$ has all three multipliers equal to $-2$, so by Lemma~\ref{lem_john} applied to each pair of fixed points, its automorphism group is $S_3$. The point $(\sigma_1,\sigma_2) = (3,3)$ corresponds to the map $f(z) = z+\frac{1}{z}$ \cite[Lemma~5.3]{Silvermana}, which has $z\mapsto -z$ as its only nontrivial automorphism.

        This completes the proof of Theorem \ref{thm_Ad} part \eqref{thm1_3}, except for the verification that the cubic is cuspidal. We defer this to Section \ref{sect_Ad_geo}.

    \subsection{Automorphism locus over $\bar{\FF}_2$} \label{subsection_autlocusf2}
        In $\bar{\FF}_2$, we still have the automorphism locus contained in $S\cup \Per_1(1)$, but equation \eqref{disc} reduces and we have the components
        \begin{align*}
            S &= V(\sigma_1^2\sigma_2-\sigma_1^2) = V(\sigma_1)\cup V(\sigma_2-1)\\
            \Per_1(1)&=V(\sigma_2-1).
        \end{align*}
        As before, the only point on $\Per_1(1)\setminus \{(0,1)\}$ that might have a non-trivial automorphism is the map with $\lambda_1=\lambda_2=\lambda_3=1$, which is $(\sigma_1,\sigma_2) = (1,1)$, or by the second part of the Normal Forms Lemma, $f(z)=z + \frac{1}{z}$. This has no nontrivial automorphisms over $\bar{\FF}_2$. Its unique fixed point is $\infty$, the other preimage of $\infty$ is $0$, and the unique preimage of $0$ in $\bar{\FF}_2$ is $1$. We would have expected $z\mapsto -z$ to be an automorphism, but it collapses to the identity map in characteristic $2$.

        Now we consider the intersection of the two components given by $\{(\sigma_1, \sigma_2) = (0,1)\}$. This map is conjugate to $f(z) = z^2+z$. Any possible automorphism of $f$ must fix the point at infinity, so must be of the form $\phi(z) = az+b$, where $a, b \in \bar{\FF}_2$ and $a \neq 0$. The equation
        $$f \circ \phi = \phi \circ f$$
        expands to
        \begin{align} \label{eq_gs_equations}
            z^2a^2 + (2b + 1)za + (b^2 + b) &= (z^2 + z)a + b.
        \end{align}
        Thus we have the following relations on $a$ and $b$:
        \begin{equation} \label{eq4}
            a^2 = a, \quad \text{and} \quad 2ab + a = a, \quad \text{and} \quad b^2 + b = b.
        \end{equation}
        Since $a$ cannot be zero, we have $a=1$ and $b = 0$, so the only automorphism is the identity.

        \begin{rmk}
        To recover a Zariski-closed automorphism locus, one can work instead with the automorphism group scheme \cite{FMV}. By definition, the automorphism group scheme of a rational map $f$ over $\bar{\FF}_2$ is a closed subgroup scheme of
        $$\PGL_2 = \Proj \bar{\FF}_2[a,b,c,d, (ad - bc)^{-1}]$$ determined by the ideal generated by the equation $f \circ \phi = \phi \circ f$, where $\phi \in \PGL_2$ is given by coordinates $a, b, c, d$. In the case of the map $f(z) = z^2 + z$, we can set $a = c = d = 1$ by the above reasoning about the fixed points, so the automorphism group scheme of $f$ is isomorphic to a closed subgroup scheme of $\bar{\FF}_2[b]$. With this identification, the group scheme structure on $\bar{\FF}_2[b]$ is just that of the additive group scheme $\mathbb{G}_{\text{a}}$, reflecting the fact that these elements of $\PGL_2$ are translations. By \eqref{eq4}, the relation determining the automorphism group scheme of $f$ is $b^2 = 0$, hence we obtain the automorphism group scheme
        \begin{equation*}
            \alpha_2 := \Spec \bar{\FF}_2[b]/(b^2).
        \end{equation*}
        This is a nontrivial closed subgroup scheme of $\mathbb{G}_{\text{a}}$, which is only possible in positive characteristic. The group scheme $\alpha_2$ has just one closed point, reflecting the fact that $f$ has only the identity automorphism, but the group scheme structure is nevertheless nontrivial. It is clear that there will be other instances in positive characteristic where the more general formulation of the automorphism group as a group scheme is needed in order to recover a Zariski closed automorphism locus.
        \end{rmk}

        It remains to investigate $S\setminus \Per_1(1) = V(\sigma_1)\setminus\{(0,1)\}$. Since this component is disjoint from $\Per_1(1)$, none of the multipliers are $1$, and so corresponding maps have three distinct fixed points, but they still has at least two equal multipliers. There is only a single point with a triple multiplier, since \eqref{3equal} reduces to $x(x - 1)^2$ and $\lambda=1$ is on $\Per_1(1)$. The point given by $\lambda=0$ again has $S_3$ as its automorphism group by Lemma \ref{lem_john}, and every other point on $V(\sigma_1)\setminus\{(0,1)\}$ has $C_2$ as its automorphism group.

        This completes the proof of Theorem \ref{thm_Ad} part \eqref{thm1_1}.

    \subsubsection{Normal form for $\mathcal{A}_2$}
        The symmetry locus in $\Rat_2$ traced out via the Normal Forms Lemma is also parameterized by the family $f_c(z) = z^2 + cz$ defined in the discussion after Theorem \ref{thm_Ad}.
        The Normal Forms Lemma sheds some light on what is happening in the family $f_c$. There are two finite fixed points with multiplier $c$, and $\infty$ is a fixed point of multiplier $0$. From this we can compute $\sigma_1=0$ and $\sigma_2=c^2$. These maps always have the order $2$ automorphism $z \mapsto z + c - 1$, which collapses to the identity when $c=1$ (giving $\alpha_2$). For a more geometric picture, the two finite fixed points are distinct, but they collapse onto each other when $c=1$.

    \subsection{Automorphism locus over $\bar{\FF}_3$}
        In $\bar{\FF}_3$, equation \eqref{disc} again reduces and we have
        \begin{align}
            S &= V(2\sigma_1^3 + \sigma_1^2\sigma_2 - \sigma_1^2 - \sigma_2^2 - 2\sigma_1\sigma_2) \label{f3locus}\\
            \Per_1(1) &= V(\sigma_2 - 2\sigma_1). \notag
        \end{align}

        Over $\bar{\FF}_3$, both $(\sigma_1,\sigma_2) = (3,3)$ and $(\sigma_1,\sigma_2) =(-6,12)$ (the triple-repeated multiplier maps) reduce to $(\sigma_1,\sigma_2) =(0,0)$, the unique intersection of the two curves. This is the only possibility for a map with all three multipliers equal, since equation \eqref{3equal} reduces to $x^3 - 1$, which factors as $(x - 1)^3$. Thus, there is no map with all three fixed points distinct and all three multipliers equal, so, by the same arguments as before, there is no map with automorphism group $S_3$.

        On the remainder of $S\setminus \Per_1(1)$, it is still true that all three fixed points are distinct and two multipliers are equal, so corresponding maps have automorphism group $C_2$. Thus the automorphism locus over $\bar{\FF}_3$ is a cuspidal cubic $S$ on which all maps have automorphism group $C_2$.

        This completes the proof of Theorem \ref{thm_Ad} part \eqref{thm1_2}, except for the verification that the cubic is cuspidal. We do this next.

    \subsection{Geometry of the Automorphism Locus} \label{sect_Ad_geo}
        For every prime $p\neq 2$, we have shown that the automorphism locus is given by a cubic. It is natural to ask if this cubic is cuspidal, as is the case in characteristic $0$, or if reduction modulo $p$ changes the geometry. We now prove the curve remains cuspidal.
        \begin{prop}\label{lem_Ad_geo}
            Let $p>2$. Then the automorphism locus $\mathcal{A}_2 \subset\mathcal{M}_2(\Fpbar)$ is a cuspidal cubic. In particular, it is irreducible. Furthermore, if $p>3$, then the cusp is the unique point with automorphism group $S_3$ and all other points have automorphism group $C_2$.
        \end{prop}
        \begin{proof}
            We first show that the automorphism locus has a unique singularity. If the locus were reducible, it would be the union of three lines or the union of a line and a degree $2$ curve. In either case, one of the tangent lines would divide the defining polynomial. So it suffices to show that the tangent lines do not divide the defining polynomial. If there is a single tangent line with multiplicity $2$, then the curve is cuspidal by definition.

            In the case where $p=3$, the automorphism locus is given by equation $\eqref{f3locus}$. The singularities are given by the common vanishing of the partial derivatives,
            \begin{align*}
                f_{\sigma_1}&=2\sigma_1\sigma_2-2\sigma_1-2\sigma_2\\
                f_{\sigma_2}&=\sigma_1^2-2\sigma_1-2\sigma_2,
            \end{align*}
            which is the single point $(\sigma_1,\sigma_2) = (0,0)$. The tangent lines at this singularity are given by the lowest-degree homogeneous component of equation $\eqref{f3locus}$, which is $-\sigma_1^2-\sigma_2^2-2\sigma_1\sigma_2=-(\sigma_1+\sigma_2)^2$. This is a double tangent line, and since $\sigma_1+\sigma_2$ does not divide $\eqref{f3locus}$, we are done.

            In the case where $p>3$, the automorphism locus is given by
            \begin{equation*}
                S = V(2\sigma_1^3 + \sigma_1^2\sigma_2 - \sigma_1^2 - 4\sigma_2^2 - 8\sigma_1\sigma_2 + 12\sigma_1 + 12\sigma_2-36).
            \end{equation*}
            The partial derivatives are
            \begin{align*}
                S_{\sigma_1}&=6\sigma_1^2+2\sigma_1\sigma_2-2\sigma_1-8\sigma_2+12,\\
                S_{\sigma_2}&=\sigma_1^2-8\sigma_2-8\sigma_1+12\sigma_2,
            \end{align*}
            and the only singularity is $(\sigma_1,\sigma_2) = (-6,12)$, which was shown above to have $S_3$ as its automorphism group. To compute the tangent lines, we need to first move the singularity to the origin with the translation $\sigma_1'=\sigma_1+6$ and $\sigma_2'=\sigma_2-12$, so then
            \begin{equation*}
                S' = V(2\sigma_1'^3+\sigma_1'^2\sigma_2'-25\sigma_1'^2-20\sigma_1'\sigma_2'-4\sigma_2'^2).
            \end{equation*}
            From this form we can see that the tangent lines are given by $$-25\sigma_1'^2-20\sigma_1'\sigma_2'-4\sigma_2'^2=-(5\sigma_1'+2\sigma_2')^2.$$ Once again, there is a double tangent line which does not divide the defining polynomial.
        \end{proof}

        This completes the proof of Theorem \ref{thm_Ad}.

\section*{Appendix}
In this Appendix, we summarize the classification of finite subgroups of $\PGL_2(k)$ up to conjugacy, where $k$ is an algebraically closed field of prime characteristic $p$. The results are essentially due to Dickson \cite{Dickson} and Beauville \cite{beauville}; for details and proofs, see Faber \cite{Faber}.

Let $B(k) \subset \PGL_2(k)$ be the \emph{Borel group}, that is, the group of affine transformations $z \mapsto \alpha z + \beta$, where $\alpha \in k^\times$ and $\beta \in k^+$. Each finite subgroup of $B(k)$ may be written as a semidirect product $\mu \rtimes \Lambda$, where $\mu$ is a finite subgroup of $k^\times$ and $\Lambda$ is a finite subgroup of $k^+$.

A finite group is called \emph{$p$-semi-elementary} if it has a unique $p$-Sylow subgroup of order $p$. Each subgroup of $B(k)$ is either $1$, a $p$-regular cyclic group of order at least $2$, or a $p$-semi-elementary group.
\begin{table}[h]
    \centering
    \begin{tabular}{c|c||c|c|c|c}
    Group               & ...where...          & $p = 2$        & $p = 3$          & $p = 5$          & $p \geq 7$ \\ \hline
    $\PGL_2(\FF_q)$     & $q$ is a power of $p$& $p$-irr            & $p$-irr             & $p$-irr           & $p$-irr        \\
    $\PSL_2(\FF_q)$     & $q$ is a power of $p$& ($p$-irr)          & $p$-irr            & $p$-irr            & $p$-irr          \\
    $1$                   &                     & $p$-reg             & $p$-reg              & $p$-reg             & $p$-reg        \\
    Cyclic $C_n$        & $n \geq 2$ and $(n,p) = 1$& $p$-reg             & $p$-reg             & $p$-reg              & $p$-reg          \\
    $p$-semi-elementary $\mu \rtimes \Lambda$
                        & $\mu \subseteq k^\times$ and $\Lambda \subseteq k^+$
                                                & $p$-irr$^{\ast}$             & $p$-irr$^{\ast}$               & $p$-irr$^{\ast}$              & $p$-irr$^{\ast}$            \\
    Dihedral $D_{2n}$      & $n = 2$               & $\emptyset$          & $p$-reg            & $p$-reg             & $p$-reg           \\
    Dihedral $D_{2n}$      & $n > 2$ and $(n,p) = 1$   & $p$-irr           & $p$-reg            & $p$-reg             & $p$-reg          \\
    Tetrahedral $A_4$      &                      & ($p$-irr)          & ($p$-irr)           & $p$-reg             & $p$-reg           \\
    Icosahedral $A_5$      &                      & ($p$-irr)          & $p$-irr             & ($p$-irr)           & $p$-reg            \\
    Octahedral $S_4$       &                      & $\emptyset$            & ($p$-irr)           & $p$-reg              & $p$-reg
    \end{tabular}
    \caption{Finite subgroups of $\PGL_2(k)$ up to conjugacy, where $k$ is an algebraically closed field of prime characteristic $p$.
    }
    \label{tab:classification}
\end{table}

Table \ref{tab:classification} should be read as follows.
\begin{itemize}
    \item The empyset symbol $\emptyset$ denotes that no finite group isomorphic to $\Gamma$ exists in $\PGL_2(k)$. For each entry that is not marked $\emptyset$, a finite group isomorphic to $\Gamma$ exists in $\PGL_2(k)$.
    \item The entry $p$-reg in row $\Gamma$ means that the group $\Gamma$ exists in $\PGL_2(k)$, is unique up to conjugacy, and $\Gamma$ is $p$-regular.
    \item The entry $p$-irr means that $\Gamma$ exists in $\PGL_2(k)$, is unique up to conjugacy, and $\Gamma$ is $p$-irregular.
    \item The entry $p$-irr$^{\ast}$ means that $\Gamma$ exists in $\PGL_2(k)$, possibly with multiple conjugacy classes, and $\Gamma$ is $p$-irregular.
    \item We mark some entries with parentheses ($p$-irr) to denote that, while the group $\Gamma$ exists in $\PGL_2(k)$, it is accounted for elsewhere in Table \ref{tab:classification} due to an accidental isomorphism. Thus, given a field $k$, each subgroup $\Gamma$ of $\PGL_2(k)$ belongs to exactly one row marked $p$-reg, $p$-irr, or $p$-irr$^{\ast}$. For the purposes of the realizability problem, it suffices to study just these cases.
\end{itemize}
The complete list of accidental isomorphisms for the entries marked ($p$-irr) is as follows.
\begin{itemize}
\item If $p = 2$ and $q$ is a power of $p$, then $\PGL_2(\FF_q) \cong \PSL_2(\FF_q)$.
\item If $p = 2$, then $A_4 \cong B(\FF_4)$ is $p$-semi-elementary.
\item If $p = 2$, then $A_5 \cong \PGL_2(\FF_4)$.
\item If $p = 3$, then $A_4 \cong \PSL_2(\FF_3)$.
\item If $p = 3$, then $S_4 \cong \PGL_2(\FF_3)$.
\item If $p = 5$, then $A_5 \cong \PSL_2(\FF_5)$.
\end{itemize}


For our applications, it is also useful to understand the possible containments between these subgroups of $\PGL_2(\bar{\FF}_p).$ By \cite[Remark 2.1]{Faber}, a finite subgroup $\Gamma \subset \PGL_2(\bar{\FF}_p)$ is $p$-semi-elementary if and only if it fixes a unique point in $\PP^1(\bar{\FF}_p).$
A subgroup is dihedral if and only if it stabilizes, but does not fix, a subset of $\PP^1(\bar{\FF}_p)$ of cardinality 2.
From these facts, and general group theory, we can deduce that for any strict inclusion of subgroups $\Gamma \subsetneq \Gamma' \subset \PGL_2(\bar{\FF}_p)$,
\begin{enumerate}
    \item If $\Gamma \cong A_4$, then $\Gamma'$ is isomorphic to $S_4, \PGL_2(\FF_q)$, or $\PSL_2(\FF_q)$ for some power $q$ of $p$.
    \item If $\Gamma \cong A_5$ or $\Gamma \cong S_4$, then $\Gamma'$ is isomorphic to $\PGL_2(\FF_q)$ or $\PSL_2(\FF_q)$ for some power $q$ of $p$.
    \item If $\Gamma \cong \PSL_2(\FF_q)$ for some power $q$ of $p$, then $\Gamma'$ is isomorphic to $\PGL_2(\FF_q')$ or $\PSL_2(\FF_q')$ for some power $q'$ of $p$.
\end{enumerate}

\end{document}